\documentclass[a4paper,12pt,leqno,oneside]{amsart}

\usepackage{amssymb}
\usepackage{mathrsfs}
\usepackage{graphicx,color}
\usepackage{newtxtext}
\usepackage[varg]{newtxmath}
\usepackage[shortlabels]{enumitem}
\usepackage{mathtools}
\mathtoolsset{showonlyrefs=true}
\usepackage{geometry}%
\newtheorem{theorem}{Theorem}[section]
\newtheorem{corollary}[theorem]{Corollary}
\newtheorem{proposition}[theorem]{Proposition}
\newtheorem{lemma}[theorem]{Lemma}
\newtheorem{definition}[theorem]{Definition}
\newtheorem{remark}[theorem]{Remark}

\DeclareMathOperator{\Div}{div}

\newcommand{\R}{\mathbb{R}}

\usepackage{constants}
\newconstantfamily{eps}{symbol=\varepsilon}
\newcommand{\Feq}{f^{\mathrm{eq}}}
\usepackage{lineno}

\usepackage{bm}
\renewcommand{\vec}[1]{\bm{#1}}
\usepackage{hyperref}

\title{Long-time Asymptotic Behavior of Nonlinear Fokker-Planck Type Equations with Periodic Boundary Conditions}

\author{Yekaterina Epshteyn}
\address[Yekaterina Epshteyn]%
{Department of Mathematics,
The University of Utah,
Salt Lake City, UT 84112, USA}
\email{epshteyn@math.utah.edu}

\author{Chun Liu}
\address[Chun Liu]%
{Department of Applied Mathematics, Illinois Institute of Technology.
Chicago, IL 60616, USA}
\email{cliu124@iit.edu}

\author{Masashi Mizuno}
\address[Masashi Mizuno]%
{Department of Mathematics, College of Science
and Technology, Nihon University, Tokyo 101-8308 JAPAN}
\email{mizuno.masashi@nihon-u.ac.jp}

\allowdisplaybreaks[1]
\begin{document}

\begin{abstract}
In this paper, we study the asymptotic behavior of a class of nonlinear Fokker-Planck type equations 
in a bounded domain with periodic boundary conditions. The system is motivated by our study of grain
boundary dynamics, especially under the non-isothermal
environments. To obtain the long time behavior of the solutions,
in particular,  the exponential decay, the kinematic structures of the
systems are investigated using novel reinterpretation of the classical
entropy method.
\end{abstract}

\maketitle

\section{Introduction}
\label{sec:1}

%


In this paper, we will discuss a specific class of Fokker-Planck systems
that can be viewed as a special type of 
generalized diffusion models in the framework of the energetic-variational
approach \cite{MR3916774, MR1607500, epshteyn2022local}. Such systems are determined by the kinematic transport of the
probability density function, the free energy functional and the dissipation
(entropy production), \cite{baierlein_1999, MR1839500}. 
In particular
we are interested in the nonlinear equations with non-homogeneous diffusion and mobility
in a finite bounded domain with periodic boundary conditions.

These Fokker-Planck type models
are motivated by the studies of  the macroscopic behavior of the systems
that involve various  fluctuations  \cite{MR987631,MR2053476,MR3932086,MR3019444,MR3485127,MR4196904,MR4218540}. 
As in our previous work, we systematically studied such Fokker-Planck type
systems as a part of grain growth models in polycrystalline materials,
e.g. \cite{DK:gbphysrev,MR2772123,MR3729587,epshteyn2021stochastic}. 

One of the important goals of such study is to develop accurate mathematical
models that take into account critical
events, such as the  grain
disappearance or nucleation,  grain boundary disappearance, facet interchange, and
splitting of unstable junctions during
coarsening process  of microstructure. For example, the recent model derived in
\cite{epshteyn2021stochastic}  under assumption of isothermal
thermodynamics can be viewed as a further
extension of a simplified 1D critical event model
 studied in \cite{DK:BEEEKT,DK:gbphysrev,MR2772123,MR3729587}.  In \cite{epshteyn2021stochastic},
we have established the long time asymptotic results of the corresponding Fokker-Planck solutions, in terms of the
joint probability density function of misorientations (a difference in
the orientation between two neighboring grains that share a grain boundary) and triple
junctions (triple junctions are where three grain
boundaries meet), as well as the relation to the
marginal probability density of misorientations. 
For an equilibrium configuration of a
boundary network, we obtained the  explicit local algebraic
relationships, the generalized Herring Condition
formula, as well as a novel relationship that connects grain boundary energy density with the geometry of
the grain boundaries that share a triple junction.

The  nonlinear Fokker-Planck
equations proposed in this work also appear as a part of our
current study of important case of non-isothermal
thermodynamics \cite{BA00323160,BA47390682,PhysRevE.94.062117}. Such
Fokker-Planck systems
are derived with applications to
macroscopic models for grain boundary dynamics in polycrystalline
materials \cite{MR4263432,CMS-Katya-Chun-Masashi,Katya-Chun-Mzn4,epshteyn2022local}. 

Most existing  mathematical analysis work of the Fokker-Planck models is
developed for the simplified linear cases only.  This is especially true
for the well-known entropy methods developed for the asymptotic analysis of such
equations, e.g. \cite{MR1842428, MR3497125,MR1812873,MR1639292}. The
classical entropy methods
rely on the specific algebraic structures of the system, under convexity assumptions 
of the potential functions and consider systems in unbounded domains.

In this paper, we study the generalized nonlinear Fokker-Planck models,
with  inhomogeneous diffusion and/or  mobility 
parameters in a bounded domain with periodic boundary conditions.
This geometric constraint is relevant to our underlying grain boundary
applications, together with the non-convexity constraints of the
potential functions. Here we want to point out
that in the paper \cite{MR4506846}, we had overlooked this crucial
assumption in our mathematical analysis.
While the results there are valid,  the mathematical analysis becomes
much more involved with less restrictive 
assumptions on the models.
This is demonstrated in our current paper. Moreover, the mathematical
analysis results in
this paper are stronger than in \cite{MR4506846} and in close agreement with the numerical
results in \cite{MR4506846} in the absence of the convexity conditions.

The paper is organized as follows. Below, we formulate the nonlinear Fokker-Planck model with the inhomogeneous diffusion and
variable mobility parameters, introduce notations and some basic lemmas needed for the later sections. In Section~\ref{sec:2}, we first illustrate
large time asymptotic analysis for the homogeneous case. In this case, the equations become the  usual linear Fokker-Planck model. 
We employ the idea of the entropy method in terms of the velocity field of the
solution. In Sections~\ref{sec:3}-\ref{sec:4}, we extend the analysis to
the Fokker-Planck model with the inhomogeneous 
 diffusion without or with
variable mobility parameters respectively. Some conclusion remarks are
given in Section~\ref{sec:5}.


Let us start with the following Fokker-Planck type of equation
subject to the periodic boundary condition on a domain
$\Omega=[0,1)^n\subset\R^n$
\begin{equation}
 \label{eq:1.Fokker-Planck}
 \left\{
  \begin{aligned}
  & \frac{\partial f}{\partial t}
   -
   \Div
   \left(
   \frac{f}{\pi(x,t)}
   \nabla
   \left(
   D(x)\log f
   +
   \phi(x)
   \right)
   \right)
   =
   0,
   &\quad
   &x\in\Omega,
   t>0, \\
 &  f(x,0)=f_0(x),&\quad
   &x\in\Omega.
  \end{aligned}
 \right.
\end{equation}
Here $D=D(x):\Omega\rightarrow\R$,
$\pi=\pi(x,t):\Omega\times[0,\infty)\rightarrow\R$ are given positive
periodic functions with respect to $\Omega$ and $\phi=\phi(x):\Omega\rightarrow\R$ is a
given periodic function with respect to $\Omega$. The periodic boundary condition for
$f$ means,
\begin{equation}
 \nabla^l f(x_{b,1},t)=\nabla^l f(x_{b,2},t),
\end{equation}
for
\begin{equation}
    \begin{split}
        x_{b,1}
        &=
        (x_1,x_2,\ldots,x_{k-1},1,x_{k+1},\ldots,x_n),
        \\
        x_{b,2}
        &=
        (x_1,x_2,\ldots,x_{k-1},0,x_{k+1},\ldots,x_n)\in\partial\Omega,
    \end{split}
\end{equation}
$t>0$ and $l=0,1,2,\ldots$. In other words, $f$ can be smoothly
extended to a
function on the entire space $\R^n$ with the condition
$f(x,t)=f(x+e_j,t)$ for $x\in\R^n$, $t>0$ and $j=1,2,\ldots,n$, where
$e_j=(0,\ldots,1,\ldots,0)$, with the 1 in the $j$th place. Note that,
the periodic boundary condition for the function $f(x,t)$ is equivalent
to the condition that $f(x,t)$ is the function on the $n$-dimensional
torus for $t>0$. The periodic function is defined in the same way. The meaning of the periodic boundary condition for the Fokker-Planck equation can be seen in \cite[\S 4.1, p.90]{MR3288096}.

The above equation can be viewed as a generalized diffusion in the
general framework of energetic variational approaches
\cite{epshteyn2022local,MR3916774,MR2165379,MR4439423,MR3021544}.
One can see this  by introducing a virtual velocity field $\vec{u}$,
\begin{equation}
 \label{eq:1.velocity}
 \vec{u}
  =
  -
  \frac{1}{\pi(x,t)}
  \nabla
  \left(
   D(x)\log f
   +
   \phi(x)
  \right)
\end{equation}
and rewrite the  \eqref{eq:1.Fokker-Planck} in a equivalent form involving the following 
kinematic continuity equation (conservation of mass):
\begin{equation}
 \label{eq:1.Continuity_Equation}
 \left\{
  \begin{aligned}
  & \frac{\partial f}{\partial t}
   +
   \Div
   (f\vec{u})  =
   0,
   &\quad
   &x\in\Omega,\quad
   t>0, \\
  & f(x,0)=f_0(x),&\quad
   &x\in\Omega.
  \end{aligned}
 \right.
\end{equation}
The form of the first equation in \eqref{eq:1.Continuity_Equation} will
motivate us to  extend various conventional  methods established for various
fluids and diffusions  to  the nonlinear Fokker-Planck
model with inhomogeneous temperature parameter $D(x)$. 

Next, from
\eqref{eq:1.Continuity_Equation} together with integration by parts and
with the periodic boundary condition, it is easy to obtain that,
\begin{equation}
 \frac{d}{dt}
  \int_\Omega
  f\,dx
  =
  \int_\Omega
  \frac{\partial f}{\partial t}\,dx
  =
  -
  \int_\Omega
  \Div (f\vec{u})\,dx
  =
  0.
\end{equation}
Therefore, if $f_0$ is a probability density function on $\Omega$, we have,
\begin{equation}
 \label{eq:1.Mass_Conservation}
  \int_\Omega
  f\,dx
  =
  \int_\Omega
  f_0\,dx
  =1.
\end{equation}

Let $F$ be a free energy defined by,
\begin{equation}
 \label{eq:1.FreeEnergy}
  F[f]
  :=
  \int_\Omega
  \left(
   D(x)f(\log f -1)+f\phi(x)
  \right)
  \,dx.
\end{equation}
Hence, we can establish the energy law for
\eqref{eq:1.Continuity_Equation}.
Hereafter, $|\cdot|$ denotes the standard Euclidean vector norm.

\begin{proposition}
 \label{prop:1.EnergyLaw}
 Let $f$ be a solution of the periodic boundary value problem
 \eqref{eq:1.Continuity_Equation}, $\vec{u}$ be the velocity vector
 defined in \eqref{eq:1.velocity}, and let $F$ be a free energy defined
 in \eqref{eq:1.FreeEnergy}.  Then, for $t>0$,
 \begin{equation}
  \label{eq:1.EnergyLaw}
   \frac{d}{dt}F[f](t)
   =
   -\int_\Omega
   \pi(x,t)
   |\vec{u}|^2f\,dx
   =:-D_{\mathrm{Dis}}[f](t).
 \end{equation}
\end{proposition}

\begin{proof}
    Note that $D(x)$ and $\phi(x)$ are independent of $t$, so direct
    computation yields
    \begin{equation}
     \label{eq:1.EnergyLaw1}
      \frac{d}{dt}F[f](t)
      =
      \int_\Omega
      (D(x)\log f+\phi(x))
      \frac{\partial f}{\partial t}
      \,dx.
    \end{equation}
    Using \eqref{eq:1.Fokker-Planck} together with periodic boundary
    condition and integration by parts, we obtain
    \begin{equation}
     \label{eq:1.EnergyLaw2}
      \begin{split}
       \int_\Omega
       (D(x)\log f+\phi(x))
       \frac{\partial f}{\partial t}
       \,dx
       &=
       -
       \int_\Omega
       (D(x)\log f+\phi(x))
       \Div(f\vec{u})
       \,dx 
       \\
       &=
       \int_\Omega
       \nabla(D(x)\log f+\phi(x))
       \cdot
       (f\vec{u})
       \,dx. 
      \end{split}    
    \end{equation}
    From \eqref{eq:1.velocity}, $\nabla(D(x)\log
    f+\phi(x))=-\pi(x,t)\vec{u}$ hence
    \begin{equation}
     \label{eq:1.EnergyLaw3}
      \int_\Omega
      \nabla(D(x)\log f+\phi(x))
      \cdot
      (f\vec{u})
      \,dx
      =
      -
      \int_\Omega
      \pi(x,t)
      |\vec{u}|^2
      f
      \,dx.
    \end{equation}
    Thus, \eqref{eq:1.EnergyLaw} is proved by combining
    \eqref{eq:1.EnergyLaw1}, \eqref{eq:1.EnergyLaw2}, and
    \eqref{eq:1.EnergyLaw3}.
\end{proof}

Throughout this paper, we assume the H\"older regularity with
$0<\alpha<1$ for coefficients $\pi(x,t)$, $D(x)$, $\phi(x)$ and initial
datum $f_0$,
\begin{equation}
    \label{eq:1.Regularity}
    \begin{split}
 &   \pi\in C_{\mathrm{per}}^{1+\alpha, (1+\alpha)/2}(\Omega\times[0,T)),\
    D\in C_{\mathrm{per}}^{2+\alpha}(\Omega),
    \\
    &
    \phi\in C_{\mathrm{per}}^{2+\alpha}(\Omega),\
    f_0\in C_{\mathrm{per}}^{2+\alpha}(\Omega),
    \end{split}
\end{equation}
where
\begin{equation*}
  C_{\mathrm{per}}^{2+\alpha}(\Omega)
  :=
  \{g\in C^{2+\alpha}(\Omega):
  g\ \text{is a periodic function on }\Omega
  \}, 
\end{equation*}
\begin{multline*}
  C_{\mathrm{per}}^{1+\alpha, (1+\alpha)/2}(\Omega\times[0,T))
  \\
  :=
  \{g\in C^{1+\alpha, (1+\alpha)/2}(\Omega\times[0,T)): 
  \\
  g(\cdot,t)\ \text{is a periodic function on }\Omega\
  \text{for}\ t>0
  \}. 
\end{multline*}

In this paper, we will consider the classical solutions of \eqref{eq:1.Fokker-Planck} defined below. In principle, they will be
smooth enough so that all the derivatives and integrations evolved in
the equations and the estimates will make sense (see
\cite{epshteyn2022local, MR1625845,MR0241822, MR1465184}).

\begin{definition}
   A periodic function in space $f=f(x,t)$ is a classical solution of
   the problem \eqref{eq:1.Fokker-Planck} in $\Omega\times[0,T)$,
   subject to the periodic boundary condition, if $f\in
   C^{2,1}(\Omega\times(0,T))\cap
   C^{1,0}(\overline{\Omega}\times[0,T))$, $f(x,t)>0$ for $(x,t)\in
   \Omega\times[0,T)$, and satisfies equation \eqref{eq:1.Fokker-Planck}
   in a classical sense.
\end{definition}
\begin{remark}
In \cite{epshteyn2022local}, we discussed the local well-posedness
of the system with the no-flux boundary condition in the H\"older space
settings. We also mentioned the local well-posedness in
\cite[Proposition 1.5]{MR4506846}. 
\textcolor{blue}{Recently, global well-posedness of the problem \eqref{eq:1.Fokker-Planck} was studied in \cite{arXiv:2502.13151}.}
Above, we recalled the definition of the
corresponding H\"older spaces but we note that the arguments in this
paper are not necessarily in the H\"older space settings. 
 \end{remark} 
Next we will prescribe those constants in terms of $C's$. associated with various quantities.
This is important for the computations and estimates in the later sections.

\begin{enumerate}
\item The bounds for the initial data:
\begin{equation}
 \label{eq:1.Initial}
  0
  <
  \Cl{const:InitMin} 
  \leq
  f_0(x,t)
  \leq
  \Cl{const:InitMax}.
\end{equation}

 \item The lower bound for the diffusion $D(x)$, $
      \Cr{const:D_Min}\geq1$ such that
      \begin{equation}
       \label{eq:1.D_Min}
	D(x)\geq \Cl{const:D_Min}.
      \end{equation}

 \item The positive bounds for the mobility:
       \begin{equation}
	\label{eq:1.PiMinMax}
	 0<
	 \Cl{const:Pi_Min}\leq \pi(x,t)\leq \Cl{const:Pi_Max}.
       \end{equation}
       
 \item The bound for the derivatives of the mobility:
       \begin{equation}
	\label{eq:1.Pi_Time}
	 |\pi_t(x,t)|\leq \Cl{const:Pi_Time},\ \text{and}
       \end{equation}
       
       \begin{equation}
	\label{eq:1.grad_Pi}
	 |\nabla \pi(x,t)|\leq \Cl{const:grad_Pi}.
       \end{equation}
       
 \item The bound of the derivative of $D(x)$:
       \begin{equation}
	\label{eq:1.grad_D}
	 |\nabla D(x)|\leq \Cl{const:grad_D}.
       \end{equation}
       
 \item The lower bound of the Hessian of $\phi$:
       \begin{equation}
	\label{eq:1.Hesse_Potential_Lower_Bounds}
	 \nabla^2\phi(x)
	 \geq
	 -\lambda I, \quad \text{for}  \quad x\in\Omega.
       \end{equation}
\end{enumerate}

It is clear that these constants/parameters will also satisfy various
relations among themselves and they are also
associated with the original system. 

\begin{lemma}
For a positive function $D(x)$, the following is true:
 \begin{equation} 
  \label{eq:1.D_Max}
  |D(x)|
   \leq
   \Cr{const:D_Min}
   +
   \sqrt{n}\Cr{const:grad_D},
 \end{equation}
 for $ \Cr{const:D_Min}$ and $\Cr{const:grad_D}$ defined earlier and $\Omega$ being a unit square/cube.
\end{lemma}

\begin{proof}
 For any $x,y\in\Omega$, we first note that
   \begin{equation}
    \label{eq:1.D_Max1}
     D(x)
     =
     D(y)
     +
     \int_0^1
     \frac{d}{d\varepsilon}
     D(\varepsilon x+(1-\varepsilon)y)
     \,d\varepsilon.
   \end{equation}
   From \eqref{eq:1.grad_D} and $\Omega$ being a unit square/cube, we
   have
   \begin{equation}
    \label{eq:1.D_Max2}
     \left|
      \frac{d}{d\varepsilon}
      D(\varepsilon x+(1-\varepsilon)y)
     \right|
     =
     \left|
      \nabla D(\varepsilon x+(1-\varepsilon)y)
      \cdot
      (x-y)
     \right|
     \leq
     \sqrt{n}
     \Cr{const:grad_D}.
   \end{equation}
   By the triangle inequality and \eqref{eq:1.D_Max1}, we get
   \begin{equation}
    \label{eq:1.D_Max3}
     |D(x)|
     \leq
     |D(y)|
     +
     \sqrt{n}\Cr{const:grad_D}.
   \end{equation}
   Therefore, \eqref{eq:1.D_Max} is deduced by taking infimum with
   respect to $y$ in \eqref{eq:1.D_Max3}.
\end{proof}

We look at an equilibrium solution of \eqref{eq:1.Fokker-Planck} in the form
\begin{equation}
 \label{eq:1.Feq}
 \Feq (x):=
  \exp\left(-\frac{\phi(x)-\Cl{const:Feq}}{D(x)}\right),
\end{equation}
where the constant $\Cr{const:Feq}$ is determined as
\begin{equation}
 \label{eq:1.Feq-PDF}
 \int_{\Omega}
  \Feq (x)\,dx
  =1.
\end{equation}

Here one can derive the   relation between  the constant
$\Cr{const:Feq}$ and the potential function $\phi$.

\begin{lemma}
 Let $\Cr{const:Feq}$ be in \eqref{eq:1.Feq}. Then
 \begin{equation}
  \label{eq:1.Est_Feq_const}
 |\Cr{const:Feq}|\leq\|\phi\|_{L^\infty(\Omega)}.
 \end{equation}
\end{lemma}

\begin{proof}
 By the mean value theorem and \eqref{eq:1.Feq-PDF}, there exists
 $x_0\in\Omega$ such that
 \begin{equation}
  \Feq(x_0)
   =
   \frac{1}{|\Omega|}
   \int_{\Omega}
   \Feq(x)\,dx
   =1.
 \end{equation}
 Thus,
 \begin{equation}
  \exp\left(-\frac{\phi(x_0)-\Cr{const:Feq}}{D(x_0)}\right)
   =
   1
 \end{equation}
 hence $\phi(x_0)=\Cr{const:Feq}$. The estimate \eqref{eq:1.Est_Feq_const}
 is easily deduced.
\end{proof}

From \eqref{eq:1.Est_Feq_const}, it follows immediately that
\begin{lemma}
 Let $\Feq$ be defined by \eqref{eq:1.Feq}. Then, for $x\in\Omega$,
 \begin{equation}
  \label{eq:1.Feq_Est}
  \exp\left(
      -\frac{2\|\phi\|_{L^\infty(\Omega)}}{\Cr{const:D_Min}}
      \right)
  \leq
  \Feq(x)
  \leq
  \exp\left(
       \frac{2\|\phi\|_{L^\infty(\Omega)}}{\Cr{const:D_Min}}
      \right).
 \end{equation}
\end{lemma}

In \cite[Proposition 1.6]{MR4506846}, we have derived the following maximum
principle which is crucial to our analysis.

\begin{proposition}
 \label{prop:1.MaximumPrinciple} 
 Let the coefficients $\pi(x,t)$, $\phi(x)$, $D(x)$, and a positive
 probability density function $f_0(x)$ satisfy the strong positivity
 \eqref{eq:1.D_Min}, \eqref{eq:1.PiMinMax}, \eqref{eq:1.Initial} and the
 H\"older regularity \eqref{eq:1.Regularity} for $0<\alpha<1$. Let $f$
 be a classical solution of \eqref{eq:1.Fokker-Planck}, then the following holds:
 \begin{equation}
  \label{eq:1.MaximumPrinciple}
  \begin{split}
&   \exp\left(
       \frac{1}{D(x)}
       \min_{y\in\Omega}
       \left(
	D(y)
	\log \frac{f_0(y)}{f^{\mathrm{eq}}(y)}
       \right)
      \right)
  f^{\mathrm{eq}}(x)
  \\
  &\quad
  \leq
  f(x,t)
  \leq  \exp\left(
       \frac{1}{D(x)}
       \max_{y\in\Omega}
       \left(
	D(y)
	\log \frac{f_0(y)}{f^{\mathrm{eq}}(y)}
       \right)
      \right)
  f^{\mathrm{eq}}(x),
  \end{split}
 \end{equation}
 for $x\in\Omega$, $t>0$. 
\end{proposition}

From Proposition \ref{prop:1.MaximumPrinciple}, we have the following
logarithm estimates. In particular, we will show that the bound of the solution
is independent of the diffusion bound $\Cr{const:D_Min}$.

\begin{corollary}
 \label{cor:1.bounds_of_log_f}
 As in the same assumptions in Proposition
 \ref{prop:1.MaximumPrinciple}, 
 there exists a positive constant $\Cl{const:log_f}$ which
 depends only on $n$, the initial datum $f_0$ (in terms of
 $\Cr{const:InitMin},\ \Cr{const:InitMax}$), the gradient of $D$ (in
 terms of $\Cr{const:grad_D}$), and the bounds of the potential
 $\|\phi\|_{L^\infty(\Omega)}$ such that 
  \begin{equation}
  \label{eq:1.bounds_of_log_f}
  |\log f(x,t)|\leq \Cr{const:log_f}.
 \end{equation}
 Moreover, $\Cr{const:log_f}$ is independent to  the diffusion bound $\Cr{const:D_Min}$.
\end{corollary}

\begin{proof}
 By \eqref{eq:1.MaximumPrinciple}, we have by \eqref{eq:1.D_Min} that,
 \begin{equation}
 \begin{split}
&   \frac{1}{D(x)}
   \min_{y\in\Omega}
   \left(
    D(y)
    \log \frac{f_0(y)}{f^{\mathrm{eq}}(y)}
   \right)
   +  
   \log f^{\mathrm{eq}}(x)   
  \\
&\quad  \leq
   \log f(x,t)
   \leq
   \left(
   \frac{1}{D(x)}
       \max_{y\in\Omega}
       \left(
	D(y)
	\log \frac{f_0(y)}{f^{\mathrm{eq}}(y)}
       \right)
  \right)
  +
  \log
  f^{\mathrm{eq}}(x).
  \end{split}
 \end{equation}
From here we have the following by using the assumption
\eqref{eq:1.D_Min}:
 \begin{equation}
 \begin{split}
   |\log f(x,t)|
 &  \leq
   \frac{1}{\Cr{const:D_Min}}
   \|D\|_{L^\infty(\Omega)}
   (
   \|\log f_0\|_{L^\infty(\Omega)}
   \\
   &\quad
   +
   \|\log \Feq\|_{L^\infty(\Omega)}
   )
   +
   \|\log \Feq\|_{L^\infty(\Omega)}.
   \end{split}
\end{equation}
 Using \eqref{eq:1.D_Max} and \eqref{eq:1.Feq_Est}, we obtain,
\begin{equation}
 \label{eq:1.Cor.1}
 \begin{split}
  &
  \frac{1}{\Cr{const:D_Min}}
  \|D\|_{L^\infty(\Omega)}
  (
  \|\log f_0\|_{L^\infty(\Omega)}
  +
  \|\log \Feq\|_{L^\infty(\Omega)}
  )
  +
  \|\log \Feq\|_{L^\infty(\Omega)} 
  \\
  &\leq
  \frac{1}{\Cr{const:D_Min}}
  (
  \Cr{const:D_Min}
  +
  \sqrt{n}\Cr{const:grad_D}
  )
  \left(
  \|\log f_0\|_{L^\infty(\Omega)}
  +
  \frac{2\|\phi\|_{L^\infty(\Omega)}}{\Cr{const:D_Min}}
  \right)
  +
  \frac{2\|\phi\|_{L^\infty(\Omega)}}{\Cr{const:D_Min}}
  \\
  &=
  \left(
  1
  +
  \frac{\sqrt{n}\Cr{const:grad_D}}{\Cr{const:D_Min}}
  \right)
  \left(
  \|\log f_0\|_{L^\infty(\Omega)}
  +
  \frac{2\|\phi\|_{L^\infty(\Omega)}}{\Cr{const:D_Min}}
  \right)
  +
  \frac{2\|\phi\|_{L^\infty(\Omega)}}{\Cr{const:D_Min}}.
 \end{split}
\end{equation}
 Since $\Cr{const:D_Min}\geq1$, we have
 \begin{equation}
 \label{eq:1.Cor.2}
  \begin{split}
   &
   \left(
   1
   +
   \frac{\sqrt{n}\Cr{const:grad_D}}{\Cr{const:D_Min}}
   \right)
   \left(
   \|\log f_0\|_{L^\infty(\Omega)}
   +
   \frac{2\|\phi\|_{L^\infty(\Omega)}}{\Cr{const:D_Min}}
   \right)
   +
   \frac{2\|\phi\|_{L^\infty(\Omega)}}{\Cr{const:D_Min}} 
   \\
   &\leq
   (
   1
   +
   \sqrt{n}\Cr{const:grad_D}
   )
   (
   \|\log f_0\|_{L^\infty(\Omega)}
   +
   2\|\phi\|_{L^\infty(\Omega)}
   )
   +
   2\|\phi\|_{L^\infty(\Omega)}
   =:
   \Cr{const:log_f}.
  \end{split} 
 \end{equation}
 Here, the constant $\Cr{const:log_f}>0$ depends only on
 $\Cr{const:InitMin}$, $\Cr{const:InitMax}$, $\Cr{const:grad_D}$ and
 $\|\phi\|_{L^\infty(\Omega)}$.
 \end{proof}

\begin{remark}\label{rk1_8}
 We emphasize that from the above corollary, we can see that
 $\Cr{const:log_f}$ can be taken uniform 
 with respect to
 $\Cr{const:D_Min}$, while the solution $f$ of the PDE
 \eqref{eq:1.Fokker-Planck} certainly depends on the diffusion $D(x)$.
 We also note that \eqref{eq:1.bounds_of_log_f} is equivalent to 
 \begin{equation}
  e^{-\Cr{const:log_f}}
   \leq
   f(x,t)
   \leq
   e^{\Cr{const:log_f}}
 \end{equation}
 for all $x\in\Omega$ and $t>0$, namely uniform lower and upper bound of
 $f$ with respect to $\Cr{const:D_Min}$. This immediately yields the
 following Harnack type estimate:
 \begin{equation}
  \frac{\max_{x\in\Omega, t>0}f(x,t)}{\min_{x\in\Omega, t>0}f(x,t)}
   \leq
   e^{2\Cr{const:log_f}}.
 \end{equation}
\end{remark}


In addition, we would need to use later the estimate \eqref{eq:1.div-grad} of
$|\Div \vec{u}|$ in terms of $|\nabla\vec{u}|$ below,
where
\begin{equation}
 |\nabla\vec{u}|^2
  =
  \sum_{k,l=1}^n
  \left(
   \frac{\partial u^{k}}{\partial x_l}
  \right)^2,\qquad
  \vec{u}=(u^1,u^2,\ldots,u^n).
\end{equation}
To obtain \eqref{eq:1.div-grad}, we use the following Jensen's inequality
\cite[V.19]{MR2467561}, \cite[Theorem 4.3]{MR0274683}.
\begin{proposition}[{\cite[V.19]{MR2467561}}]
 Let f be a function from $\R$ to $(-\infty,\infty]$. Then $f$ is convex if and only if, 
 \begin{equation}
  \label{eq:1.Jensen}
   f(\lambda_1 x_1+\cdots+\lambda_m x_m)
   \leq
   \lambda_1 f(x_1)+\cdots+\lambda_m f(x_m)
 \end{equation}
 whenever $\lambda_1,\ldots,\lambda_m\geq0$,
 $\lambda_1+\cdots+\lambda_m=1$, and $x_1,\ldots,x_m\in\R$.
\end{proposition}
Applying Jensen's inequality \eqref{eq:1.Jensen} to convex function
$x^2$, we have that,
\begin{equation}
 \label{eq:1.Jensen_mean}
 \left(
 \frac{a_1+\cdots+a_n}{n}
 \right)^2
 \leq
 \frac{1}{n}
 \left(
  a_1^2+\cdots+a_n^2
 \right)
\end{equation}
for $a_1,\ldots,a_n\in\R$. Using, $a_j=\frac{\partial u^j}{\partial
x_j}$ in \eqref{eq:1.Jensen_mean}, we obtain,
\begin{equation}
 \label{eq:1.div-grad}
 |\Div\vec{u}|^2
  \leq
  n\sum_{j=1}^n
  \left(
   \frac{\partial u^j}{\partial x_j}
  \right)^2
  \leq
  n|\nabla\vec{u}|^2.
\end{equation}



\section{Homogeneous diffusion case}
\label{sec:2}

In this section, we consider the cases with both diffusion and mobility
being homogeneous. In this case, $D$ is a constant, and without loss of
generality, we choose $\pi\equiv1$ and take $\Omega=[0,1)^n\subset\R^n$.
 
Here, we study the following evolution equation with periodic boundary
conditions.

\begin{equation}
 \label{eq:2.FokkerPlanck}
 \left\{
  \begin{aligned}
&   \frac{\partial f}{\partial t}
   +
   \Div
   \left(
   f\vec{u}
   \right)
   =
   0,
   &\quad
   &x\in\Omega,\quad
   t>0, \\
 &  \vec{u}
   =
   -
   \nabla
   \left(
   D\log f
   +
   \phi(x)
   \right),
   &\quad
   &x\in\Omega,\quad
   t>0, \\
&   f(x,0)=f_0(x),&\quad
   &x\in\Omega.
  \end{aligned}
 \right.
\end{equation}
The free energy $F$ associated with \eqref{eq:2.FokkerPlanck} will take the form :
\begin{equation}
 \label{eq:2.FreeEnergy}
  F[f](t)
  =
  \int_\Omega
  (Df(x,t)(\log f(x,t)-1)+f(x,t)\phi(x))
  \,dx.
\end{equation}
As we had presented in Proposition \ref{prop:1.EnergyLaw}, the following energy law will hold
for any solution $f$ of \eqref{eq:2.FokkerPlanck}:
\begin{equation}
 \label{eq:2.EnergyLaw}
  \frac{d}{dt}F[f](t)
  =
  -
  \int_{\Omega}
  |\vec{u}|^2f\,dx
  =
  -D_{\mathrm{dis}}[f](t).
\end{equation}

In this setting, equation \eqref{eq:2.FokkerPlanck} is  a linear  parabolic equation. One can
use various established methods and techniques  to investigate the long-time asymptotic behavior for a
solution of \eqref{eq:2.FokkerPlanck}. For instance, making the change
of variable
\begin{equation*}
 f(x,t)=g(x,t)
  \exp\left(-\frac{\phi(x)}{D}\right),
\end{equation*}
one may associate the original equation with  a self-adjoint operator $Lg=D\Delta g-\nabla\phi\cdot\nabla
g$ on $L^2(\Omega,e^{-\frac{\phi}{D}}\,dx)$ (See
\cite{epshteyn2021stochastic}).

In this paper, we plan to study the long time behavior of the solutions, especially the 
exponential decay through the investigation of higher order time derivative of the free energy functional.
We want to point out that the current method is related to the entropy method that had been
developed previously for various Fokker-Planck type of equations in
unbounded domains. We consider here bounded domain and our approach takes the full advantage of the kinematic structures, such as looking at the velocity variable $\vec{u}$.

\begin{theorem}
 \label{thm:2}
Given the potential $\phi$ and the domain $\Omega$ in $\R^n$ in
 \eqref{eq:2.FokkerPlanck}, for any $\gamma>0$, there exists
 $\Cr{const:D_Min}\geq1$ which depends on $n$, the potential $\phi$ (in
 terms of the lower bound of the Hessian $\lambda$ defined in
 \eqref{eq:1.Hesse_Potential_Lower_Bounds} and
 $\|\phi\|_{L^\infty(\Omega)}$), the bounds of the initial data
 $\Cr{const:InitMin}$, $\Cr{const:InitMax}$ defined on
 \eqref{eq:1.Initial}, and $\gamma$, such that if \eqref{eq:1.D_Min}
 holds and,
 \begin{equation}
  \label{eq:2.Initial_Free_Energy}
   \int_\Omega
  |\nabla (D\log f_0+\phi(x))|^2
  f_0\,dx
  =:
  \Cr{const:2.InitialEnergy}<\infty,
 \end{equation}
 then, the following dissipation rate decays exponentially in time:
 \begin{equation}
  \label{eq:2.Exponential_Decay_Free_Energy}
  \int_\Omega
   |\vec{u}|^2
   f\,dx
   \leq
   \Cl{const:2.InitialEnergy}
   e^{-\gamma t}.
 \end{equation}
\end{theorem}

\begin{remark}
 The above theorem states that for given $\phi$ and the domain $\Omega$, one can obtain any decay rate $\gamma$ by choosing the diffusion constant $D$ large enough. Conversely,  from the 
 proof of the theorem, we can show that for large enough $D$ (estimate
 is given in \eqref{C_3bounds} later), the dissipation rate of the system will decay
 exponentially in time.
\end{remark}

\begin{remark}
 In the conventional entropy method, the function $\phi$ is usually
 assumed to be convex, and the domain is the whole space.  In this
 paper, with a bounded domain subject to the periodic boundary
 conditions, we can assume the Hessian of $\phi$ to be bounded below.
 In particular, we can treat cases with non-convex function $\phi$.
\end{remark}

To prove Theorem \ref{thm:2}, we first recall the second derivative of
the free energy \cite{MR4506846}.

\begin{proposition}%
 [{\cite[Proposition 2.9]{MR4506846}}]
 For the free energy  $F$  defined in  \eqref{eq:2.FreeEnergy}, and  $f$
 be a solution of \eqref{eq:2.FokkerPlanck}, we obtain: 
\begin{equation}
  \label{eq:2.Second_Derivative_Free_Energy}
  \frac{d^2}{dt^2}F[f](t)
   =
   2
   \int_\Omega
   (\nabla^2\phi(x)\vec{u}\cdot\vec{u})f\,dx
   +
   2
   \int_\Omega
   D|\nabla \vec{u}|^2
   f
   \,dx.
 \end{equation}
\end{proposition}

Since $\nabla^2\phi$ might be non-positive, our plan is to use the  second term, 
 $\int_{\Omega} D|\nabla\vec{u}|^2f \, dx$ to control the right hand side of \eqref{eq:2.Second_Derivative_Free_Energy}. 
 We first establish the following 
inequality.

\begin{lemma}
 \label{lem:2.Sobolev}
 Let $f$ be a solution of \eqref{eq:2.FokkerPlanck}, and let $\vec{u}$
 be given as in \eqref{eq:2.FokkerPlanck}.
 There exists $\Cl{const:2.Sobolev}>0$ which depends only on $n$, $f_0$ (in
 terms of $\Cr{const:InitMin}$, $\Cr{const:InitMax}$ defined in
 \eqref{eq:1.Initial}), and $\|\phi\|_{L^\infty(\Omega)}$ such that
 \begin{equation}
  \label{eq:2.SobolevType}
  \left(
   \int_\Omega
   |\vec{u}|^{p^*}f\,dx
  \right)^{\frac{1}{p^*}}
  \leq
  \Cr{const:2.Sobolev}
  \left(
   \int_\Omega
   |\nabla\vec{u}|^{2}f\,dx
  \right)^{\frac{1}{2}},
 \end{equation}
 where 
 $2< p^*<\infty$ for $n=1,2$,
 and $\frac{1}{p^*}=\frac{1}{2}-\frac{1}{n}$ for $n\geq3$.
\end{lemma}

\begin{proof}
 By the Sobolev inequality (\cite[Lemma 7.12,
 7.16]{MR1814364},\cite[p.313]{MR2759829}), there is
 $\Cl{const:Sobolev}>0$ which depends only on $n$ such that,
 \begin{equation*}
  \left(
  \int_\Omega
  |\vec{v}-\overline{\vec{v}}|^{p^*}\,dx
  \right)^{\frac{1}{p^*}}
  \leq
  \Cr{const:Sobolev}
  \left(
  \int_\Omega
  |\nabla\vec{v}|^{2}\,dx
  \right)^{\frac{1}{2}}
 \end{equation*}
 for $\vec{v}\in W^{1,2}(\Omega)$, where $\overline{\vec{v}}$ is the
 integral mean of $\vec{v}$, namely
 \begin{equation*}
  \overline{\vec{v}}
   =
   \frac{1}{|\Omega|}
   \int_\Omega
   \vec{v}\,dx.
 \end{equation*}
Here,
 $2< p^*<\infty$ for $n=1,2$,
 and $\frac{1}{p^*}=\frac{1}{2}-\frac{1}{n}$ for $n\geq3$.
 \par Since $\vec{u}$ has the scalar potential, we can compute,
 \begin{equation}
 \begin{split}
  \int_\Omega
   \vec{u}
   \,dx
 &  =
   -
   \int_\Omega
   \nabla
   \left(
   D\log f
   +
   \phi(x)
   \right)
   \,dx
   \\
   &
   =
   -
   \int_{\partial\Omega}
   \left(
   D\log f
   +
   \phi(x)
   \right)
   \nu
   \,d\sigma
   =
   \vec{0}.
   \end{split}
 \end{equation}
 Hence we obtain by taking $\vec{v}=\vec{u}$ that,
 \begin{equation*}
  \begin{split}
   \left(
   \int_\Omega
   |\vec{u}|^{p^*}f\,dx
   \right)^{\frac{1}{p^*}}
   &\leq
   \left(
   \max_{(x,t)\in\Omega\times[0,\infty)}f
   \right)^{\frac{1}{p^*}}
   \left(
   \int_\Omega
   |\vec{u}|^{p^*}\,dx
   \right)^{\frac{1}{p^*}}
   \\
   &\leq
   \Cr{const:Sobolev}
   \left(
   \max_{(x,t)\in\Omega\times[0,\infty)}f
   \right)^{\frac{1}{p^*}}
   \left(
   \int_\Omega
   |\nabla\vec{u}|^{2}\,dx
   \right)^{\frac{1}{2}}
   \\
   &\leq
   \Cr{const:Sobolev}
   \frac{(\max_{(x,t)\in\Omega\times[0,\infty)}f)^{\frac{1}{p^*}}}{(\min_{(x,t)\in\Omega\times[0,\infty)}f)^{\frac{1}{2}}}
   \left(
   \int_\Omega
   |\nabla\vec{u}|^{2}f\,dx
   \right)^{\frac{1}{2}}.
  \end{split} 
 \end{equation*}
 Note from \eqref{eq:1.bounds_of_log_f} (see Remark~\ref{rk1_8}) that $e^{-\Cr{const:log_f}}\leq
 f(x,t)\leq e^{\Cr{const:log_f}}$ for all $x\in\Omega$ and $t>0$. This
 inequality yields
 \begin{equation*}
  \frac{(\max_{(x,t)\in\Omega\times[0,\infty)}f)^{\frac{1}{p^*}}}{(\min_{(x,t)\in\Omega\times[0,\infty)}f)^{\frac{1}{2}}}
   \leq
   e^{\Cr{const:log_f}\left(
		       \frac{1}{p^*}+\frac{1}{2}
		      \right)}.
 \end{equation*}
 Letting
 $\Cr{const:2.Sobolev}=\Cr{const:Sobolev}e^{\Cr{const:log_f}\left(
 \frac{1}{p^*}+\frac{1}{2} \right)}$, we obtain
 \eqref{eq:2.SobolevType}.
\end{proof}

\begin{remark}
 In this section, $D$ is constant, which means $\nabla D=0$ hence
 $\Cr{const:log_f}$ depends on
 the $L^\infty$ norms of the initial data $f_0$ in term of $\Cr{const:InitMin}$, $\Cr{const:InitMax}$ and potential 
 $\|\phi\|_{L^\infty(\Omega)}$,  but is independent of $\Cr{const:grad_D}$
 defined in \eqref{eq:1.grad_D}.
\end{remark}

\begin{remark}
In this proof, it is important that $\vec{u}$ is a potential
 gradient. For \eqref{eq:1.Fokker-Planck}, when the mobility $\pi(x,t)$
 is not constant, the velocity $\vec{u}$ is not a potential gradient
 anymore, and we will have to derive more estimates as we will do in
 Section \ref{sec:4}.
\end{remark}

From here, we can show the following Poincare type inequality for the
velocity $\vec{u}$.

\begin{lemma}
 \label{lem:2.Poincare} Let $f$ be a solution of
 \eqref{eq:2.FokkerPlanck}, and let $\vec{u}$ be given as in
 \eqref{eq:2.FokkerPlanck}.  There exists a constant
 $\Cl{const:2.Poincare}>0$ which depends only on $n$, the $L^\infty$ norms
 of the initial data $f_0$ (in terms of $\Cr{const:InitMin}$,
 $\Cr{const:InitMax}$ defined on \eqref{eq:1.Initial}), and
 $\|\phi\|_{L^\infty(\Omega)}$ such that,
 \begin{equation}
  \label{eq:2.PoincareType}
  \int_\Omega
   |\vec{u}|^2f\,dx
   \leq
   \Cr{const:2.Poincare}
   \int_\Omega
   |\nabla\vec{u}|^2f\,dx.
 \end{equation}
\end{lemma}

\begin{proof}
 We show \eqref{eq:2.PoincareType} for the case $n\geq 3$.  First, we
 use the H\"older inequality and \eqref{eq:1.Mass_Conservation} that
 \begin{equation}
  \left(
   \int_\Omega
   |\vec{u}|^2f\,dx
  \right)^{\frac{1}{2}}
  \leq
  \left(
   \int_\Omega
   |\vec{u}|^{p^*}f\,dx
  \right)^{\frac{1}{p^*}}
  \left(
   \int_\Omega
   f\,dx
  \right)^{\frac{1}{2}-\frac{1}{p^*}}
  =
  \left(
   \int_\Omega
   |\vec{u}|^{p^*}f\,dx
  \right)^{\frac{1}{p^*}}
 \end{equation}
 where we select $p^*>2$ to satisfy
 $\frac{1}{p^*}=\frac{1}{2}-\frac{1}{n}$ so we can employ the
 Sobolev's inequality  \eqref{eq:2.SobolevType}. Next, we use
 \eqref{eq:2.SobolevType} to have that,
 \begin{equation}
  \left(
   \int_\Omega
   |\vec{u}|^{p^*}f\,dx
  \right)^{\frac{1}{p^*}}
  \leq
  \Cr{const:2.Sobolev}
  \left(
   \int_\Omega
   |\nabla\vec{u}|^{2}f\,dx
  \right)^{\frac{1}{2}}.
 \end{equation}
 Letting $\Cr{const:2.Poincare}=\Cr{const:2.Sobolev}^2$, we obtain
 \eqref{eq:2.PoincareType}. Note that $\Cr{const:log_f}$ depends on
 $\Cr{const:InitMin}$, $\Cr{const:InitMax}$ and
 $\|\phi\|_{L^\infty(\Omega)}$ so we can take $\Cr{const:2.Poincare}$
 which depends only on $n$, $\Cr{const:InitMin}$,
 $\Cr{const:InitMax}$ and $\|\phi\|_{L^\infty(\Omega)}$.

 For the cases $n=1,2$,  \eqref{eq:2.SobolevType} holds for any
 $2<p^*<\infty$. If, for instance,  by  taking  $p^*=6$, we will be able to use
 the exact same estimates as in the case $n=3$ above to get the result.
\end{proof}

Now, we are in a position to derive the time derivative of the integral
$|\vec{u}|^2$.

\begin{proposition}
 Let $f$ be a solution of \eqref{eq:2.FokkerPlanck}, and let $\vec{u}$
 be given as in \eqref{eq:2.FokkerPlanck}. For any $\gamma>0$, there
 exists $\Cr{const:D_Min}\geq1$ which depends only on $n$,
 $\lambda$, $\gamma$, $\Cr{const:InitMin}$, $\Cr{const:InitMax}$ defined
 on \eqref{eq:1.Initial}, and $\|\phi\|_{L^\infty(\Omega)}$ such that
 \begin{equation}
  \label{eq:2.Differential_Inequality_for_the_second_derivative_of_the_Free_Energy}
  \frac{d}{dt}
   \left(
    \int_\Omega|\vec{u}|^2
    f\,dx
   \right)
    \leq
    -\gamma
    \int_\Omega
    |\vec{u}|^2
    f\,dx.
 \end{equation}
\end{proposition}

\begin{proof}
 From the second derivative of the free energy
 \eqref{eq:2.Second_Derivative_Free_Energy} and the lower bounds of
 $\nabla^2\phi$ \eqref{eq:1.Hesse_Potential_Lower_Bounds}, we have
 \begin{equation}
  \frac{d^2}{dt^2}F[f](t)
   \geq
   -2\lambda
   \int_\Omega|\vec{u}|^2f\,dx
   +
   2D
   \int_\Omega|\nabla\vec{u}|^2f\,dx.
 \end{equation}
 Next, we use the Poincare type inequality \eqref{eq:2.PoincareType}
 that
 \begin{equation}
  \int_\Omega|\nabla\vec{u}|^2f\,dx
   \geq
   \frac{1}{\Cr{const:2.Poincare}}
   \int_\Omega|\vec{u}|^2f\,dx.
 \end{equation}
 Together with the two inequalities, we obtain
 \begin{equation}
  \frac{d^2}{dt^2}F[f](t)
   \geq
   \left(
    -2\lambda
    +
    \frac{2D}{\Cr{const:2.Poincare}}
   \right)
   \int_\Omega|\vec{u}|^2f\,dx.
 \end{equation}
 Choose $\Cr{const:D_Min}\geq1$ sufficiently large such that
 \begin{equation}
 \label{C_3bounds}
  -2\lambda
   +
   \frac{2\Cr{const:D_Min}}{\Cr{const:2.Poincare}}
   \geq
   \gamma.
 \end{equation}
 Then, from the energy law \eqref{eq:2.EnergyLaw}, we have that,
 \begin{equation}
  -\frac{d}{dt}
   \left(
    \int_\Omega
    |\vec{u}|^2f
    \,dx
   \right)
   \geq
   \gamma
   \int_\Omega
   |\vec{u}|^2f
   \,dx.
 \end{equation}
 This finishes the proof of the Proposation.
\end{proof}
\par Finally, we are in position to prove the main result of this Section, Theorem~\ref{thm:2}.
\begin{proof}%
 [Proof of Theorem \ref{thm:2}]
 Applying the Gronwall inequality to the differential inequality
 \eqref{eq:2.Differential_Inequality_for_the_second_derivative_of_the_Free_Energy}, we obtain
 \begin{equation}
  \int_\Omega
   |\vec{u}|^2f
   \,dx
   \leq
   e^{-\gamma t}
   \int_\Omega
   |
   \nabla
   \left(
    D\log f_0
    +
    \phi(x)
    \right)
   |^2f_0
   \,dx.
 \end{equation}
 Using \eqref{eq:2.Initial_Free_Energy}, we obtain
 \eqref{eq:2.Exponential_Decay_Free_Energy}.
\end{proof}

In Section \ref{sec:2}, in the case that $D$ and $\pi$ are constants, we
derive the exponential decay of the time derivative of the free energy
$F$ by construction of the differential inequality
\eqref{eq:2.Differential_Inequality_for_the_second_derivative_of_the_Free_Energy}. The
crucial idea is to use the Poincare type inequality
\eqref{eq:2.PoincareType} in Lemma \ref{lem:2.Poincare}.
In the following sections, the Poincare type inequality will also play crucial roles in the study.

\begin{remark}
 We emphasize that the above argument is based on the kinematic structure of the equation
 \eqref{eq:1.Continuity_Equation} and the velocity $\vec{u}$. This is in consistent with the conventional
 entropy methods in the whole domain cases, which work on the whole PDE itself. This approach is much 
 straightforward and uses the physical relevant quantities.
\end{remark}

\begin{remark}
 The result in Theorem \ref{thm:2} shows that for a given decay rate
 $\gamma>0$, we may find the class of equations \eqref{eq:2.FokkerPlanck}
 with sufficiently large diffusion $D$ and bounded initial data $f_0$ (in
 terms of $D_{\mathrm{dis}}[f](0)<\infty$), then the system will evolve
 exponentially towards equilibrium.  In particularly, we prove the
 exponential decay property of dissipation rate
 $D_{\mathrm{dis}}[f](t)$, which appears in \eqref{eq:2.EnergyLaw}, with
 the given rate $\gamma$.  Conversely the proof of the theorem also give
 the sufficient condition in $D$ for the exponential decay for a system
 with bounded initial data. Moreover it is open question whether the result holds
 \emph{without assumption for the diffusion coefficient
 $D\geq\Cr{const:D_Min}$ to be large enough}.
\end{remark}


\section{Inhomogeneous diffusion case}
\label{sec:3}

%
In this section, we consider the following
evolution equation with inhomogeneous diffusion while the mobility remaining a constant.
\begin{equation}
 \label{eq:3.FokkerPlanck}
 \left\{
  \begin{aligned}
 &  \frac{\partial f}{\partial t}
   +
   \Div
   \left(
   f\vec{u}
   \right)
   =
   0,
   &\quad
   &x\in\Omega,\quad
   t>0, \\
 &  \vec{u}
   =
   -
   \nabla
   \left(
   D(x)\log f
   +
   \phi(x)
   \right),
   &\quad
   &x\in\Omega,\quad
   t>0, \\
 &  f(x,0)=f_0(x),&\quad
   &x\in\Omega,
  \end{aligned}
 \right.
\end{equation}
with periodic boundary conditions. Without loss of generality, we take
$\Omega=[0,1)^n\subset\R^n$. We choose a strictly positive periodic
function $D=D(x)$ with the positive lower bound, $\Cr{const:D_Min}\geq1$,
i.e., $ D(x)\geq \Cr{const:D_Min}, $ for $x\in\Omega$.

Again, the direct formal computations show that the free energy $F$ and the basic energy law \eqref{eq:1.EnergyLaw} take
the following specific forms,
\begin{equation}
 \label{eq:3.FreeEnergy}
  F[f]
  :=
  \int_\Omega
  \left(
   D(x)f(\log f -1)+f\phi(x)
  \right)
  \,dx,
\end{equation}
and
\begin{equation}
 \label{eq:3.EnergyLaw}
  \frac{d}{dt}F[f](t)
  =
  -\int_\Omega
  |\vec{u}|^2f\,dx
  =:
  - D_{\mathrm{dis}}[f](t).
\end{equation}

As discussed in \cite{MR4506846}, the energy law \eqref{eq:3.EnergyLaw}
with the free energy \eqref{eq:3.FreeEnergy} carries all the physics of
the system, and the system, together with the kinematic assumption of
$f$, will yield the equation \eqref{eq:3.FokkerPlanck} by the energetic variational
approach. In particular, we note that inhomogeneity $D(x)$ is the
source of the nonlinearity in \eqref{eq:3.FokkerPlanck}.  Like what we had proved
in the previous section, Theorem \ref{thm:2}, the following theorem
states that for any function $\phi$ with bounded second derivatives, we
can show exponential decay for the dissipation rate in the energy law
provided the diffusion coefficient $D(x)$ is sufficiently
large. However, besides assumptions made in Theorem \ref{thm:2}, we
will make additional assumptions on the gradients of $D(x)$ and
$\phi$, as well as assumption on the number of dimension $n$.

\begin{theorem}
 \label{thm:3}
 Assume $n=1,2,3$. For a fixed constant $\gamma>0$, there exist positive
 constants $\Cr{const:D_Min}\geq1$ and $\Cl{const:3.Initial_Energy}>0$ which
 depend on dimension $n$, the lower bound of Hessian of $\phi$ in
 terms of $\lambda$ defined in
 \eqref{eq:1.Hesse_Potential_Lower_Bounds}, $\gamma$, the initial data
 $f_0$ (in terms of $\Cr{const:InitMin}$, $\Cr{const:InitMax}$ defined
 on \eqref{eq:1.Initial}), the gradient of $D(x)$ (in terms of
 $\Cr{const:grad_D}$ defined in \eqref{eq:1.grad_D}),
 $\|\phi\|_{L^\infty(\Omega)}$ , and
 $\|\nabla\phi\|_{L^\infty(\Omega)}$, such that if \eqref{eq:1.D_Min}
 holds and
 \begin{equation}
  \label{eq:3.Initial_Free_Energy}
   \int_\Omega |\nabla (D(x)\log f_0+\phi(x))|^2f_0\,dx
   \leq
   \Cr{const:3.Initial_Energy},
 \end{equation}
 then,  for $t>0$, we have the following result:
 \begin{equation}
  \label{eq:3.Exponential_Decay_Free_Energy}
  \int_\Omega
   |\vec{u}|^2
   f\,dx
   \leq
   \Cr{const:3.Exponential_Coefficient}
   e^{-\gamma t},
 \end{equation}
 for  a positive constant
 $\Cl{const:3.Exponential_Coefficient}>0$.
\end{theorem}

As in section \ref{sec:2}, we have the following Sobolev-type inequality
for the velocity $\vec{u}$. Although the proof is the same as in Lemma
\ref{lem:2.Sobolev}, due to the fact that $\vec{u}$ is a potential
gradient in this case, the constant $\Cr{const:3.Sobolev}$, which is in
Lemma \ref{lem:3.Sobolev}, will depend on $\Cr{const:grad_D}$ defined in
\eqref{eq:1.grad_D}.

\begin{lemma}
 \label{lem:3.Sobolev} 
 Let $f$ be a solution of \eqref{eq:3.FokkerPlanck}, and let $\vec{u}$
 be given as in \eqref{eq:3.FokkerPlanck}.  There exists
 $\Cr{const:3.Sobolev}>0$ which depends only on $n$, $f_0$ (in terms of
 $\Cr{const:InitMin}$, $\Cr{const:InitMax}$ defined in
 \eqref{eq:1.Initial}), the gradient of $D(x)$ (in terms of
 $\Cr{const:grad_D}$ defined in \eqref{eq:1.grad_D}), and
 $\|\phi\|_{L^\infty(\Omega)}$ such that
 \begin{equation}
  \label{eq:3.SobolevType}
  \left(
   \int_\Omega
   |\vec{u}|^{p^*}f\,dx
  \right)^{\frac{1}{p^*}}
  \leq
  \Cl{const:3.Sobolev}
  \left(
   \int_\Omega
   |\nabla\vec{u}|^{2}f\,dx
  \right)^{\frac{1}{2}},
 \end{equation}
 where $2< p^*<\infty$ for $n=1,2$ and
 $\frac{1}{p^*}=\frac{1}{2}-\frac{1}{n}$ for $n\geq3$.
\end{lemma}

With the extra dependence of the constants, the same proof in Lemma \ref{lem:2.Poincare}
will yield the following Poincare-type inequality for the velocity $\vec{u}$.

\begin{lemma}
 \label{lem:3.Poincare} 
 Let $f$ be a solution of \eqref{eq:3.FokkerPlanck}, and let $\vec{u}$
 be given as in \eqref{eq:3.FokkerPlanck}. There exists a constant
 $\Cl{const:3.Poincare}>0$ which depends only on $n$, $f_0$ (in terms of
 $\Cr{const:InitMin}$, $\Cr{const:InitMax}$ defined in
 \eqref{eq:1.Initial}), the gradient of $D(x)$ (in terms of
 $\Cr{const:grad_D}$ defined in \eqref{eq:1.grad_D}), and
 $\|\phi\|_{L^\infty(\Omega)}$ such that
 \begin{equation}
  \label{eq:3.PoincareType}
  \int_\Omega
   |\vec{u}|^2f\,dx
   \leq
   \Cr{const:3.Poincare}
   \int_\Omega
   |\nabla\vec{u}|^2f\,dx.
 \end{equation}
\end{lemma}


\begin{remark}
 Compare to Lemma \ref{lem:2.Poincare}, the constants
 $\Cr{const:3.Sobolev}$ and $\Cr{const:3.Poincare}$ depend not
 only on $n$, $\Cr{const:InitMin}$, $\Cr{const:InitMax}$, and
 $\|\phi\|_{L^\infty(\Omega)}$, but also $\Cr{const:grad_D}$, the
 upper bound of the gradient of $D(x)$. We emphasize
 again that the constants above $\Cr{const:3.Sobolev}$ and
 $\Cr{const:3.Poincare}$ are uniform with respect to
 $\Cr{const:D_Min}$.

\end{remark}

In the next lemma, we obtain the following interpolation inequality.

\begin{lemma}
 \label{lem:3.Sobolev.cubic}
 Let $n=1,2,3$. Let $f$ be a solution of \eqref{eq:3.FokkerPlanck}, and let
 $\vec{u}$ be given as in \eqref{eq:3.FokkerPlanck}. 
 Then we have
 \begin{equation}
  \label{eq:3.Sobolev.cubic}
   \int_\Omega|\vec{u}|^3f\,dx
   \leq
   \frac{3}{4}
   \Cr{const:3.Sobolev}^{\frac{3}{2}}
   \int_\Omega|\nabla \vec{u}|^2f\,dx
   +
   \frac{1}{4}
   \Cr{const:3.Sobolev}^{\frac{3}{2}}
   \left(
   \int_\Omega|\vec{u}|^2f\,dx
   \right)^3.
 \end{equation}
\end{lemma}

\begin{proof}
The proof follows from the same as those in \cite[Lemma 3.14]{MR4506846}
(See also Lemma \ref{lem:4.Sobolev.cubic}) together with Lemma
\ref{lem:3.Sobolev}.
\end{proof}

We next recall the second derivative of the free energy \cite{MR4506846}, which can be obtained by 
direct computation from the original system \eqref{eq:3.FokkerPlanck}:

\begin{proposition}%
 [{\cite[Proposition 3.13]{MR4506846}}]
 Let $f$ be a solution of \eqref{eq:3.FokkerPlanck} and let $\vec{u}$ be given as in
 \eqref{eq:3.FokkerPlanck}. Then the following is true:
 \begin{equation}
  \label{eq:3.Second_Derivative_Free_Energy}
   \begin{split}
    \frac{d^2}{dt^2}F[f](t)
    &=
    2\int_\Omega ((\nabla^2\phi(x)) \vec{u}\cdot\vec{u}) f\,dx
    +
    2\int_\Omega D(x)|\nabla \vec{u}|^2 f\,dx
    \\
    &\qquad
    -
    \int_\Omega
    (\log f-1)
    \nabla |\vec{u}|^2\cdot\nabla D(x) f
    \,dx
    \\
    &\qquad
    -
    2
    \int_\Omega (1+\log f)\vec{u}\cdot\nabla D(x)\Div\vec{u}f\,dx
    \\
    &\qquad
    +
    2
    \int_\Omega
    \frac{1}{D(x)}|\vec{u}|^2
    \log f\left(\vec{u}\cdot\nabla D(x)\right) f\,dx
    \\
    &\qquad
    +
    2
    \int_\Omega
    \frac{1}{D(x)}
    (\log f)^2\left(\vec{u}\cdot\nabla D(x)\right)^2 f\,dx
    \\
    &\qquad
    +
    2
    \int_\Omega
    \frac{1}{D(x)}
    \log f\left(\vec{u}\cdot\nabla D(x)\right)\left(\vec{u}\cdot\nabla \phi(x)\right) f\,dx.
   \end{split}
 \end{equation}
\end{proposition}

Below out of total 7 terms on the right-hand side of
\eqref{eq:3.Second_Derivative_Free_Energy}, we first consider the 3rd,
4th, and 7th integral of the right-hand side of
\eqref{eq:3.Second_Derivative_Free_Energy}.  Notice the 2nd and 6th
terms are with the right positive sign, while the 5th term includes the
cubic order of $\vec{u}$.

\begin{lemma}
 Let $f$ be a solution of \eqref{eq:3.FokkerPlanck}, and let $\vec{u}$ be given as
 in \eqref{eq:3.FokkerPlanck}. Then, 
 we have
 \begin{equation}
  \label{eq:3.Second_Derivative_Free_Energy.1}
  \begin{split}
   \frac{d^2}{dt^2}F[f](t)
   &\geq
   2\int_\Omega
   \biggl(
   (\nabla^2\phi(x)\vec{u}\cdot\vec{u})
   \\
   &\qquad\qquad
   -
   \Bigl(
   \frac{1}{2}
   +
   \frac{\Cr{const:log_f}\Cr{const:grad_D}}{\Cr{const:D_Min}}
   \|\nabla \phi\|_{L^\infty(\Omega)}
   \Bigr) |\vec{u}|^2
   \biggr) f\,dx
   \\
   &\quad
   +
   2\int_\Omega
   \biggl(
   1
   -
   \frac{(1+n)\Cr{const:grad_D}^2(\Cr{const:log_f}+1)^2}{\Cr{const:D_Min}}
   \biggr)
   D(x)|\nabla \vec{u}|^2 f\,dx
   \\
   &\quad
   +
   2
   \int_\Omega
   \frac{1}{D(x)}|\vec{u}|^2
   \log f\left(\vec{u}\cdot\nabla D(x)\right) f\,dx.
  \end{split}
 \end{equation}
\end{lemma}

\begin{proof}
 Note that the integrand of the 6th term of
 \eqref{eq:3.Second_Derivative_Free_Energy} is non-negative, we need to
 estimate the integrands of the 3rd, 4th, and 7th integrals of
 \eqref{eq:3.Second_Derivative_Free_Energy}.
 We have by Cauchy's inequality with  a fixed $\varepsilon=1/4$
 (See \cite[p.662]{MR1625845}) that,
 \begin{equation}
  \begin{split}
   |
   (\log f-1)\nabla |\vec{u}|^2\cdot\nabla D(x) f
   |
   &\leq
   2
   (|\log f|+1)^2
   |\nabla \vec{u}|^2
   |\nabla D(x)|^2
   f
   +
   \frac{1}{2}
   |\vec{u}|^2 f
   \\
   &\leq
   \frac{2\Cr{const:grad_D}^2(\Cr{const:log_f}+1)^2}{\Cr{const:D_Min}}
   D(x)
   |\nabla \vec{u}|^2
   f
   +
   \frac{1}{2}
   |\vec{u}|^2 f,
  \end{split}
 \end{equation}
 \begin{equation}
  \begin{split}
   |2(1+\log f)\vec{u}\cdot\nabla D(x)\Div\vec{u}f|
   &\leq
   2n
   (|\log f|+1)^2
   |\nabla \vec{u}|^2
   |\nabla D(x)|^2 f
   +
   \frac{1}{2}
   |\vec{u}|^2 f
   \\
   &\leq
   \frac{2n\Cr{const:grad_D}^2(\Cr{const:log_f}+1)^2}{\Cr{const:D_Min}}
   D(x)
   |\nabla \vec{u}|^2
   f
   +
   \frac{1}{2}
   |\vec{u}|^2 f,
  \end{split}
 \end{equation}
 and
 \begin{equation}
  \left|
   \frac{2}{D(x)}
   \log f\left(\vec{u}\cdot\nabla D(x)\right)\left(\vec{u}\cdot\nabla \phi(x)\right) f
  \right|
  \leq
  \frac{2\Cr{const:log_f}\Cr{const:grad_D}}{\Cr{const:D_Min}}
  \|\nabla \phi\|_{L^\infty(\Omega)}
  |\vec{u}|^2f.
 \end{equation}
 Here we used \eqref{eq:1.D_Min}, \eqref{eq:1.grad_D},
 \eqref{eq:1.bounds_of_log_f}, and \eqref{eq:1.div-grad}.  Thus, using
 all inequalities above in \eqref{eq:3.Second_Derivative_Free_Energy},
 we arrive at the estimate \eqref{eq:3.Second_Derivative_Free_Energy.1}.
\end{proof}


Next, we compute the term with cubic order of $\vec{u}$ in \eqref{eq:3.Second_Derivative_Free_Energy.1}.

\begin{lemma}
 Let $f$ be a solution of \eqref{eq:3.FokkerPlanck}, and let $\vec{u}$ be given as
 in \eqref{eq:3.FokkerPlanck}. Assume \eqref{eq:1.D_Min}.
 Then, we have
 \begin{equation}
  \label{eq:3.Second_Derivative_Free_Energy.2}
   \begin{split}
    &
    \left|
    2
    \int_\Omega
    \frac{1}{D(x)}|\vec{u}|^2
    \log f\left(\vec{u}\cdot\nabla D(x)\right) f\,dx
    \right|
    \\
    &\leq
    \frac{3\Cr{const:log_f}\Cr{const:grad_D}\Cr{const:3.Sobolev}^{\frac{3}{2}}}{2\Cr{const:D_Min}}
    \int_\Omega D(x)|\nabla\vec{u}|^2f\,dx
    +
    \frac{\Cr{const:log_f}\Cr{const:grad_D}\Cr{const:3.Sobolev}^{\frac{3}{2}}}{2\Cr{const:D_Min}}
    \left(
    \int_\Omega |\vec{u}|^2f\,dx
    \right)^3.
   \end{split}
 \end{equation}
\end{lemma}

\begin{proof}
By Using \eqref{eq:1.D_Min}, \eqref{eq:1.grad_D}, and
 \eqref{eq:1.bounds_of_log_f}, we have
 \begin{equation}
  \label{eq:3.Second_Derivative_Free_Energy.2-1}
   \left|
    2
    \int_\Omega
    \frac{1}{D(x)}|\vec{u}|^2
    \log f\left(\vec{u}\cdot\nabla D(x)\right) f\,dx
   \right|
   \leq
   \frac{2\Cr{const:log_f}\Cr{const:grad_D}}{\Cr{const:D_Min}}
   \int_\Omega|\vec{u}|^3f\,dx.
 \end{equation}
 Next, employing the interpolation inequality \eqref{eq:3.Sobolev.cubic}
 and we arrive at:
 \begin{equation}
  \label{eq:3.Second_Derivative_Free_Energy.2-2}
   \begin{split}
    \frac{2\Cr{const:log_f}\Cr{const:grad_D}}{\Cr{const:D_Min}}
    \int_\Omega|\vec{u}|^3f\,dx
  &  \leq
    \frac{3\Cr{const:log_f}\Cr{const:grad_D}\Cr{const:3.Sobolev}^{\frac{3}{2}}}{2\Cr{const:D_Min}}
    \int_\Omega|\nabla\vec{u}|^2f\,dx
    \\
    &\qquad
    +
    \frac{\Cr{const:log_f}\Cr{const:grad_D}\Cr{const:3.Sobolev}^{\frac{3}{2}}}{2\Cr{const:D_Min}}
    \left(
    \int_\Omega |\vec{u}|^2f\,dx
    \right)^3.
   \end{split}
 \end{equation}
 Finally combining \eqref{eq:3.Second_Derivative_Free_Energy.2-1} with
 \eqref{eq:3.Second_Derivative_Free_Energy.2-2} and using $D(x)\geq \Cr{const:D_Min} \geq1$,
 we obtain the result \eqref{eq:3.Second_Derivative_Free_Energy.2}.
\end{proof}

Next, we study the integrals involving $\nabla\phi$ and $\nabla^2\phi$ in 
\eqref{eq:3.Second_Derivative_Free_Energy}, by using the positive 2nd term.
We will show that with large diffusion bound $\Cr{const:D_Min}$, we can reduce the original 
\eqref{eq:3.Second_Derivative_Free_Energy} into a specific form.

\begin{proposition}
 \label{prop:3.Second_Derivative_Free_Energy}
 Let $f$ be a solution of \eqref{eq:3.FokkerPlanck}, and let $\vec{u}$ be given as
 in \eqref{eq:3.FokkerPlanck}. Assume the lower bound of the diffusion $D(x)$, $\Cr{const:D_Min}\geq1$ is large enough such that the following condition holds:
 \begin{equation}
  \label{eq:3.Second_Derivative_Free_Energy.3.Assumption}
   \max
   \left\{
   3\Cr{const:log_f}\Cr{const:grad_D}\Cr{const:3.Sobolev}^{\frac{3}{2}},
   2\Cr{const:log_f}\Cr{const:grad_D}\|\nabla\phi\|_{L^\infty(\Omega)},
   4(1+n)(\Cr{const:log_f}+1)^2\Cr{const:grad_D}^2
   \right\}
   \leq
   \Cr{const:D_Min},
 \end{equation}
 where the constant $\Cr{const:log_f}$ is the bound of the solution and
 $\Cr{const:grad_D}$ is the bound of the gradient of $D(x)$, as defined in
 \eqref{eq:1.bounds_of_log_f} and \eqref{eq:1.grad_D} in Section \ref{sec:1}.
 Then we obtain,
 \begin{equation}
 \label{eq:3.Second_Derivative_Free_Energy.3}
 \begin{split}
  \frac{d^2F}{dt^2}[f](t)
 &  \geq
   -2(\lambda+1)
   \int_\Omega
   |\vec{u}|^2
   f
   \,dx
   +
   \int_\Omega
   D(x)|\nabla \vec{u}|^2 f\,dx
   \\
   &\qquad
   -\frac{1}{6}
   \left(
   \int_\Omega
   |\vec{u}|^2
   f
   \,dx
   \right)^3.
   \end{split}
 \end{equation}
\end{proposition}

\begin{remark}
 We again point out the fact  that $\Cr{const:log_f}$ and $\Cr{const:3.Sobolev}$
 are uniformly bounded with respect to $\Cr{const:D_Min}$. The constants
 $\Cr{const:grad_D}$ and $\|\nabla\phi\|_{L^\infty(\Omega)}$ are
 independent of $\Cr{const:D_Min}$, so
  for given $n$, $f_0$, $\phi$,
 \eqref{eq:3.Second_Derivative_Free_Energy.3.Assumption1},
 \eqref{eq:3.Second_Derivative_Free_Energy.3.Assumption3}, and
 \eqref{eq:3.Second_Derivative_Free_Energy.3.Assumption4} below hold for
 sufficiently large $\Cr{const:D_Min}$.
\end{remark}

\begin{proof}
 Note that the condition \eqref{eq:3.Second_Derivative_Free_Energy.3.Assumption} yields the following inequalities:
 \begin{equation}
  \label{eq:3.Second_Derivative_Free_Energy.3.Assumption1}
   \frac{\Cr{const:log_f}\Cr{const:grad_D}\Cr{const:3.Sobolev}^{\frac{3}{2}}}{\Cr{const:D_Min}}
   \leq
   \frac{1}{3},
 \end{equation}
 \begin{equation}
  \label{eq:3.Second_Derivative_Free_Energy.3.Assumption3}
   \frac{\Cr{const:log_f}\Cr{const:grad_D}}{\Cr{const:D_Min}}
   \|\nabla\phi\|_{L^\infty(\Omega)}
   \leq
   \frac{1}{2},
 \end{equation}
 and
 \begin{equation}
  \label{eq:3.Second_Derivative_Free_Energy.3.Assumption4}
   (\Cr{const:log_f}+1)^2
   \Cr{const:grad_D}^2
   \leq
   \frac{1}{4(1+n)}\Cr{const:D_Min}.
 \end{equation}
 
 First, combining \eqref{eq:3.Second_Derivative_Free_Energy.1} and
 \eqref{eq:3.Second_Derivative_Free_Energy.2}, we obtain
 \begin{equation}
  \label{eq:3.Second_Derivative_Free_Energy.3.1}
  \begin{split}
   \frac{d^2F}{dt^2}[f](t)
   &\geq
   2\int_\Omega
   [
   (\nabla^2\phi(x)\vec{u}\cdot\vec{u})
   \\
   &\qquad\qquad
   -
   \biggl(
   \frac{1}{2}
   +
   \frac{\Cr{const:log_f}\Cr{const:grad_D}}{\Cr{const:D_Min}}
   \|\nabla \phi\|_{L^\infty(\Omega)}
  \biggr ) |\vec{u}|^2
   ] f\,dx
   \\
   &\qquad
   +
   2\int_\Omega
   \biggl(
   1
   -
   \frac{(1+n)\Cr{const:grad_D}^2(\Cr{const:log_f}+1)^2}{\Cr{const:D_Min}}
   \biggr)
   D(x)|\nabla \vec{u}|^2 f\,dx
   \\
   &\qquad
   -
   \frac{3\Cr{const:log_f}\Cr{const:grad_D}\Cr{const:3.Sobolev}^{\frac{3}{2}}}{2\Cr{const:D_Min}}
   \int_\Omega D(x)|\nabla\vec{u}|^2f\,dx
   \\
  &\qquad
   -
   \frac{\Cr{const:log_f}\Cr{const:grad_D}\Cr{const:3.Sobolev}^{\frac{3}{2}}}{2\Cr{const:D_Min}}
   \left(
   \int_\Omega |\vec{u}|^2f\,dx
   \right)^3.
  \end{split}
 \end{equation}

 We consider the integral of $|\vec{u}|^2f$ in
 \eqref{eq:3.Second_Derivative_Free_Energy.3.1}. Using
 \eqref{eq:1.Hesse_Potential_Lower_Bounds}, we get,
 \begin{equation}
  (\nabla^2\phi(x)\vec{u}\cdot\vec{u})
   \geq
    -
    \lambda
    |\vec{u}|^2.
 \end{equation}
 Then applying
 the assumption
 \eqref{eq:3.Second_Derivative_Free_Energy.3.Assumption3}, we have
 \begin{equation}
  \label{eq:3.Second_Derivative_Free_Energy.3.2}
  \begin{split}
&   2\int_\Omega
   \biggl(
   (\nabla^2\phi(x)\vec{u}\cdot\vec{u})
   -
   \Bigl(
   \frac{1}{2}
   +
   \frac{\Cr{const:log_f}\Cr{const:grad_D}}{\Cr{const:D_Min}}
   \|\nabla \phi\|_{L^\infty(\Omega)}
   \Bigr) |\vec{u}|^2
   \biggr) f\,dx
    \\
 &\quad
   \geq
   -2(\lambda+1)
   \int_\Omega
   |\vec{u}|^2
   f
   \,dx.
   \end{split}
 \end{equation}
 Next we consider the integral of $D(x)|\nabla\vec{u}|^2f$ in
 \eqref{eq:3.Second_Derivative_Free_Energy.3.1}. 
 Using \eqref{eq:3.Second_Derivative_Free_Energy.3.Assumption1}, we
 obtain,
 \begin{equation}
  \label{eq:3.Second_Derivative_Free_Energy.3.3}
  \begin{split}
&   2
   \biggl(
   1
   -
   \frac{(1+n)\Cr{const:grad_D}^2(\Cr{const:log_f}+1)^2}{\Cr{const:D_Min}}
   \biggr)
   -
   \frac{3\Cr{const:log_f}\Cr{const:grad_D}\Cr{const:3.Sobolev}^{\frac{3}{2}}}{2\Cr{const:D_Min}}
   \\
   &\quad
   \geq
   2
   \biggl(
   1
   -
   \frac{(1+n)\Cr{const:grad_D}^2(\Cr{const:log_f}+1)^2}{\Cr{const:D_Min}}
    \biggr)
    -\frac{1}{2}.
    \end{split}
 \end{equation}
 Thus, by 
 \eqref{eq:3.Second_Derivative_Free_Energy.3.Assumption4} that
 \begin{equation}
  \label{eq:3.Second_Derivative_Free_Energy.3.4}
   \frac{(1+n)\Cr{const:grad_D}^2(\Cr{const:log_f}+1)^2}{\Cr{const:D_Min}}
   \leq
   \frac{1}{4}.
 \end{equation}
 Combining \eqref{eq:3.Second_Derivative_Free_Energy.3.3} and
 \eqref{eq:3.Second_Derivative_Free_Energy.3.4}, we arrive at
 \begin{equation}
  \label{eq:3.Second_Derivative_Free_Energy.3.5}
  \begin{split}
   2\int_\Omega
   \biggl(
   1
  & -
   \frac{(1+n)\Cr{const:grad_D}^2(\Cr{const:log_f}+1)^2}{\Cr{const:D_Min}}
   \biggr)
   D(x)|\nabla \vec{u}|^2 f\,dx
   \\
 &  -
   \frac{3\Cr{const:log_f}\Cr{const:grad_D}\Cr{const:3.Sobolev}^{\frac{3}{2}}}{2\Cr{const:D_Min}}
   \int_\Omega D(x)|\nabla\vec{u}|^2f\,dx
   \geq
   \int_\Omega
   D(x)|\nabla \vec{u}|^2 f\,dx.
   \end{split}
 \end{equation}
 Finally, use \eqref{eq:3.Second_Derivative_Free_Energy.3.Assumption1}
 for the coefficient of the last term in the right-hand side of \eqref{eq:3.Second_Derivative_Free_Energy.3.1}, we have
 \begin{equation}
  \label{eq:3.Second_Derivative_Free_Energy.3.6}
   \frac{\Cr{const:log_f}\Cr{const:grad_D}\Cr{const:3.Sobolev}^{\frac{3}{2}}}{2\Cr{const:D_Min}}
   \leq
   \frac{1}{6}.
 \end{equation}
 With the fact that $\Cr{const:log_f}$ and $\Cr{const:3.Sobolev}$ are
independent of $\Cr{const:D_Min}$, 
 combining \eqref{eq:3.Second_Derivative_Free_Energy.3.2},
 \eqref{eq:3.Second_Derivative_Free_Energy.3.5}, and
 \eqref{eq:3.Second_Derivative_Free_Energy.3.6}, we obtain
 \eqref{eq:3.Second_Derivative_Free_Energy.3}.
\end{proof}

Using Proposition \ref{prop:3.Second_Derivative_Free_Energy}
and with suitably large $\Cr{const:D_Min}$, 
we further reduce inequality  the dissipation rate functional \eqref{eq:3.Second_Derivative_Free_Energy.3}.

\begin{lemma}
 \label{lem:3.Differential_Inequality_Dissipation}
 Let $f$ be a solution of \eqref{eq:3.FokkerPlanck}, and let $\vec{u}$
 be given as in \eqref{eq:3.FokkerPlanck}.  Then for a given $\gamma>0$,
 there exists a sufficiently large positive constant
 $\Cr{const:D_Min}\geq1$ which depends only on $n$,
 $\|\phi\|_{L^\infty(\Omega)}$, $\|\nabla\phi\|_{L^\infty(\Omega)}$,
 $\lambda$ from \eqref{eq:1.Hesse_Potential_Lower_Bounds}(the
 lower bound of the Hessian of $\phi$), $\Cr{const:InitMin}$,
 $\Cr{const:InitMax}$ defined in \eqref{eq:1.Initial}(the bounds of the
 initial datum $f_0$), $\Cr{const:grad_D}$ defined in
 \eqref{eq:1.grad_D}(the bound of the gradient of $D$),
 $\Cr{const:3.Sobolev}$ appeared in \eqref{eq:3.SobolevType}, and
 $\gamma$ such that if \eqref{eq:1.D_Min} holds, then,
 \begin{equation}
  \label{eq:3.Differential_Inequality_Dissipation}
   \frac{d^2}{dt^2}F[f](t)
   \geq
   \gamma\int_{\Omega}
   |\vec{u}|^2
   f\,dx
   -
   \frac{1}{6}
   \left(
   \int_\Omega
   |\vec{u}|^2
   f
   \,dx
   \right)^3.
 \end{equation}
\end{lemma}

\begin{proof}
 By Lemma \ref{lem:3.Poincare} together with \eqref{eq:1.D_Min},
 \begin{equation}
  \label{eq:3.Differential_Inequality_Dissipation1}
   \begin{split}
   &
   -2(\lambda+1)
   \int_\Omega
   |\vec{u}|^2
   f
   \,dx
   +
   \int_\Omega
   D(x)|\nabla \vec{u}|^2 f\,dx
    \\
   &\geq
   -2(\lambda+1)
   \int_\Omega
   |\vec{u}|^2
   f
   \,dx
   +
   \Cr{const:D_Min}
   \int_\Omega
   |\nabla \vec{u}|^2 f\,dx
   \\
   &\geq
   \left(
   -2(\lambda+1)
   +
   \frac{\Cr{const:D_Min}}{\Cr{const:3.Poincare}}
   \right)
   \int_\Omega
   |\vec{u}|^2
   f
   \,dx.
   \end{split}
 \end{equation}
 Note that $\Cr{const:3.Poincare}$ depends only on $n$,
 $\Cr{const:InitMin}$, $\Cr{const:InitMax}$ defined in
 \eqref{eq:1.Initial}, $\Cr{const:grad_D}$ defined in
 \eqref{eq:1.grad_D}, and $\|\phi\|_{L^\infty(\Omega)}$. 
 Thus, first we take large $\Cr{const:D_Min}\geq1$ such that the
 assumptions \eqref{eq:3.Second_Derivative_Free_Energy.3.Assumption1},
 \eqref{eq:3.Second_Derivative_Free_Energy.3.Assumption3}, and
 \eqref{eq:3.Second_Derivative_Free_Energy.3.Assumption4} hold. Next,
 for $\gamma>0$, take $\Cr{const:D_Min}\geq1$ further sufficiently large
 such that
 \begin{equation}
  -2(\lambda+1)
   +
   \frac{\Cr{const:D_Min}}{\Cr{const:3.Poincare}}
 \geq
 \gamma.
 \end{equation}
 Then, we can use
 \eqref{eq:3.Second_Derivative_Free_Energy.3} 
 and \eqref{eq:3.Differential_Inequality_Dissipation1},
 hence we obtain
 \eqref{eq:3.Differential_Inequality_Dissipation}.
\end{proof}

Next, we will recall the following Gronwall-type inequality from
\cite{MR4506846}.

\begin{lemma}
 [{Lemma 3.16 in \cite{MR4506846}}]
 \label{lem:3.Gronwall}
  Let $c,d,p>0$ be positive constants, such that $p>1$. Let
 $g:[0,\infty)\rightarrow\R$ be a non-negative function that satisfies the
 following differential inequality,
 \begin{equation}
  \label{eq:3.Gronwall}
  \frac{dg}{dt}\leq -cg+dg^p.
 \end{equation}
 If
 \begin{equation}
  g(0)< \left(
	 \frac{c}{d}
	\right)^{\frac{1}{p-1}},
 \end{equation}
 then,  we obtain for $t>0$,
 \begin{equation}
  \label{eq:3.54}
  g(t)
   \leq
   \left(
    g(0)^{-p+1}-\frac{d}{c}
   \right)^{-\frac{1}{p-1}}
   e^{-ct}.
 \end{equation}
\end{lemma}

This Gronwall-type inequality will allow us to show  the exponential decay of
the dissipation rate of the  free energy in \eqref{eq:3.EnergyLaw},
the main result of this Section.

\begin{proof}
 [Proof of Theorem \ref{thm:3}] 
 For any $\gamma>0$, using Lemma
 \ref{lem:3.Differential_Inequality_Dissipation}, take sufficiently
 large positive number $\Cr{const:D_Min}\geq1$. Then with specifically
 the following quantities,
 \eqref{eq:3.Differential_Inequality_Dissipation} becomes
 \eqref{eq:3.Gronwall}:
 \begin{equation*}
  g(t)=-\frac{d}{dt}F[f](t)
  =
  \int_\Omega 
  |\vec{u}|^2
  f\,dx,\quad
  c=\gamma,\quad
  p=3,\quad
  \text{and}
  \quad
  d=\frac{1}{6}.  
 \end{equation*} 
 Thus, if
 \begin{equation}
  g(0)
   =
  \int_\Omega
   |\nabla (D(x)\log f_0(x)+\phi(x))|^2
   f_0\,dx
   <
   \left(
    6\gamma
   \right)^{\frac{1}{2}},
 \end{equation}
 then by Lemma \ref{lem:3.Gronwall}
 \begin{equation}
    \begin{split}
        g(t)
        &=
        \int_\Omega
        |\vec{u}|^2f
        \,dx
        \\
        &\leq
        \left(
            \left(
            \int_\Omega
            |\nabla (D(x)\log f_0(x)+\phi(x))|^2
            f_0\,dx
            \right)^{-2}
            -
        \frac{1}{6\gamma}
        \right)^{-\frac{1}{2}}
        e^{-\gamma t}.
    \end{split}
 \end{equation}
Taking
 $\Cr{const:3.Initial_Energy}=(6\gamma)^{\frac{1}{2}}$ and
 \begin{equation}
  \Cr{const:3.Exponential_Coefficient}
   =
   \left(
    \left(
     \int_\Omega
     |\nabla (D(x)\log f_0(x)+\phi(x))|^2
     f_0\,dx
    \right)^{-2}
    -
    \frac{1}{6\gamma}
   \right)^{-\frac{1}{2}},
 \end{equation}
we will arrive at the conclusion of  Theorem \ref{thm:3}.
\end{proof}

In this Section, we had demonstrated the exponential decay of the time
derivative of the free energy in the case that $D(x)$ is inhomogeneous
and the mobility is constant. In the next section, we will consider the
case of inhomogeneous diffusion $D(x)$ and variable mobility $\pi(x,t)$.


\section{Inhomogeneous diffusion case with variable mobility}
\label{sec:4}

This section will be devoted to the
following nonlinear Fokker-Planck equation with inhomogeneity in both
diffusion $D(x)$ and mobility $\pi(x,t)$, which are bounded periodic
positive functions defined in a bounded domain $\Omega$ in the Euclidean
space of $n$-dimension.

\begin{equation}
 \label{eq:4.FokkerPlanck}
 \left\{
  \begin{aligned}
   \frac{\partial f}{\partial t}
   &+
   \Div
   \left(
   f\vec{u}
   \right)
   =
   0,
   &\quad
   &x\in\Omega,\quad
   t>0, \\
   \vec{u}
   &=
   -
   \frac{1}{\pi(x,t)}
   \nabla
   \left(
   D(x)\log f
   +
   \phi(x)
   \right),
   &\quad
   &x\in\Omega,\quad
   t>0, \\
   f(&x, 0)=f_0(x),&\quad
   &x\in\Omega.
  \end{aligned}
 \right.
\end{equation}
Again, without loss of generality, we take $\Omega=[0,1)^n\subset\R^n$.
For the convenience of the readers, we recall (as defined in Section 1)
that the periodic function $D(x)$ is bounded from below with the
constant $\Cr{const:D_Min}\geq1$, and the periodic function $\pi(x,t)$
is bounded both from below and above by the positive constants
$\Cr{const:Pi_Min}$ and $\Cr{const:Pi_Max}$, namely
\begin{equation*}
 D(x)\geq\Cr{const:D_Min},\quad
  \Cr{const:Pi_Min}\leq \pi(x,t)\leq \Cr{const:Pi_Max}
\end{equation*}
for any $x\in\Omega$ and $t>0$.

The free energy $F$ and the basic energy law \eqref{eq:1.EnergyLaw}
still takes the standard form:
\begin{equation}
 \label{eq:4.FreeEnergy}
  F[f]
  :=
  \int_\Omega
  \left(
   D(x)f(\log f -1)+f\phi(x)
  \right)
  \,dx,
\end{equation}
and the system satisfies the following energy dissipation law:
\begin{equation}
 \label{eq:4.EnergyLaw}
  \frac{dF}{dt}[f](t)
  =
  -\int_\Omega
  \pi(x,t)
  |\vec{u}|^2f\,dx
  =:
  - D_{\mathrm{dis}}[f](t).
\end{equation}

We will first state the main result of this section: for a given system in \eqref{eq:4.FokkerPlanck},
with suitable conditions of the initial data and the mobility, under relatively mild inhomogeneity conditions,
 one can find a diffusion $D(x)$ that
is large enough, such that the system will convergence exponentially
to a equilibrium.

\begin{theorem}
 \label{thm:4}
 Assume $n=1,2,3$. For a fixed constant  $\gamma>0$, 
 there exist positive constants
 $\Cr{const:D_Min},$ 
 $\Cr{const:Pi_Time},$ 
 $\Cr{const:grad_Pi}$ 
and
 $\Cl{const:4.Initial_Energy}$, which  depend only on the given constant $\gamma$, $n$, the potential  $\phi$ (in terms of 
 Hessian bound $\lambda$
 defined in \eqref{eq:1.Hesse_Potential_Lower_Bounds},
 $\|\phi\|_{L^\infty(\Omega)}$, and $\|\nabla\phi\|_{L^\infty(\Omega)}$), 
 the bound of initial data $\Cr{const:InitMin}$, $\Cr{const:InitMax}$ defined in \eqref{eq:1.Initial}, the bound $\Cr{const:grad_D}$ of $\nabla D(x)$
 defined in \eqref{eq:1.grad_D}, and  the bounds for the mobility $\Cr{const:Pi_Min}$,
 $\Cr{const:Pi_Max}$ defined in \eqref{eq:1.PiMinMax},
 such that if \eqref{eq:1.D_Min}, \eqref{eq:1.Pi_Time},
 \eqref{eq:1.grad_Pi} hold and,
 \begin{equation}
  \label{eq:4.Initial_Free_Energy}
  \int_\Omega
  \pi(x,0)
  |\nabla D\log f_0+\phi(x)|^2\,dx
  \leq
  \Cr{const:4.Initial_Energy},
 \end{equation}
 then  for $t>0$, the following is true,
 \begin{equation}
  \label{eq:4.Exponential_Coefficient}
  \int_\Omega
  \pi(x,t)
  |\vec{u}|^2
   f\,dx
   \leq
   \Cr{const:4.Exponential_Coefficient}
   e^{-\gamma t},
 \end{equation}
 for a  positive constant
 $\Cl{const:4.Exponential_Coefficient}>0$.
\end{theorem}

\begin{remark}
 Unlike Sections 2 and  3, the velocity $\vec{u}$ in this section is not a gradient field,
 which means that  $\overline{\vec{u}}\neq0$. We will need to re-evaluate  those estimates obtained in the previous sections
 and if necessary, derive the ones under the new conditions.
 Here we start with  the Sobolev-type inequality for $\vec{u}$ with the mean value
 $\overline{\vec{u}}$.
\end{remark}

\begin{lemma}
 \label{lem:4.Sobolev} Let $n$ be a natural number, 
 for $\varepsilon>0$,   
 the positive constants $\Cr{const:Pi_Min}, \Cr{const:grad_Pi}$ being defined as the bounds of the mobility
 \eqref{eq:1.PiMinMax} and the bound of the derivative of the mobility
 \eqref{eq:1.grad_Pi} respectively.
 Let $f$ be a solution of \eqref{eq:4.FokkerPlanck}, and let $\vec{u}$
 be given as in \eqref{eq:4.FokkerPlanck}. 
 If
 \begin{equation}
  \label{eq:4.Sobolev_Assumption}
  \Cr{const:grad_Pi}
   \leq
   \Cr{const:Pi_Min}
   \left(
   \frac{\varepsilon}{2}\right)^{\frac{1}{2}},
 \end{equation}
 then 
 there exists $\Cl{const:4.Sobolev}>0$ which depends
 only on $n$, $\pi$ (in terms of $\Cr{const:Pi_Min}$,
 $\Cr{const:Pi_Max}$ defined in \eqref{eq:1.PiMinMax}), $f_0$ (in terms
 of $\Cr{const:InitMin}$, $\Cr{const:InitMax}$ defined on
 \eqref{eq:1.Initial}), the gradient of $D(x)$ (in terms of
 $\Cr{const:grad_D}$ defined in \eqref{eq:1.grad_D}), and
 $\|\phi\|_{L^\infty(\Omega)}$ such that,
 \begin{equation}
  \label{eq:4.SobolevType}
  \left(
   \int_\Omega
   |\vec{u}|^{p^*}f\,dx
  \right)^{\frac{1}{p^*}}
  \leq
  \Cr{const:4.Sobolev}
  \left(
   \int_\Omega
   \left(
    2|\nabla\vec{u}|^{2}
    +
    \varepsilon|\vec{u}|^{2}
   \right)
   f\,dx
  \right)^{\frac{1}{2}},
 \end{equation}
 where $2< p^*<\infty$ for $n=1,2$, and
 $\frac{1}{p^*}=\frac{1}{2}-\frac{1}{n}$ for $n\geq3$.
\end{lemma}

\begin{proof}
 First, using the definition of $\Cr{const:Pi_Min}$, the positive bounds for the mobility \eqref{eq:1.PiMinMax}, we have,
 \begin{equation}
  \begin{split}
   \left(
   \int_\Omega
   |\vec{u}|^{p^*}f\,dx
   \right)^{\frac{1}{p^*}}
   &=
   \left(
   \int_\Omega
   \left|
   \frac{1}{\pi(x,t)}
   \nabla
   \left(
   D(x)\log f
   +
   \phi(x)
   \right)
   \right|^{p^*}f\,dx
   \right)^{\frac{1}{p^*}}
   \\
   &\leq
   \frac{1}{\Cr{const:Pi_Min}}
   \left(
   \int_\Omega
   \left|
   \nabla
   \left(
   D(x)\log f
   +
   \phi(x)
   \right)
   \right|^{p^*}f\,dx
   \right)^{\frac{1}{p^*}}.
  \end{split}
 \end{equation}
 Then, we can use Lemma \ref{lem:3.Sobolev} that
 \begin{equation}
 \begin{split}
&  \left(
   \int_\Omega
   \left|
   \nabla
   \left(
   D(x)\log f
   +
   \phi(x)
   \right)
   \right|^{p^*}f\,dx
   \right)^{\frac{1}{p^*}}
    \\
&\quad   \leq
   \Cr{const:3.Sobolev}
   \left(
    \int_\Omega
    |
    \nabla^2
    \left(
     D(x)\log f
     +
     \phi(x)
    \right)
    |^{2}f\,dx
   \right)^{\frac{1}{2}}.
   \end{split}
 \end{equation}
 Using the definition of $\Cr{const:Pi_Min}$ and $\Cr{const:Pi_Max}$ in
 the positive bounds for the mobility \eqref{eq:1.PiMinMax}, we obtain
 \begin{equation}
  \left(
   \int_\Omega
   |\vec{u}|^{p^*}f\,dx
  \right)^{\frac{1}{p^*}}
  \leq
  \frac{\Cr{const:Pi_Max}\Cr{const:3.Sobolev}}{\Cr{const:Pi_Min}}
   \left(
    \int_\Omega
    \left|
    -\frac{1}{\pi(x,t)}
    \nabla^2
    \left(
     D(x)\log f
     +
     \phi(x)
    \right)
    \right|^{2}
    f\,dx
   \right)^{\frac{1}{2}}.
 \end{equation}
 Here
 \begin{equation}
    \begin{split}
    \nabla\vec{u}
    &=
    -\frac{1}{\pi(x,t)}
    \nabla^2
    \left(
     D(x)\log f
     +
     \phi(x)
    \right)
    \\
    &\quad
    +
    \frac{1}{\pi^2(x,t)}
    \nabla\pi(x,t)
    \otimes
    \nabla
    \left(
     D(x)\log f
     +
     \phi(x)
    \right).
    \end{split}
 \end{equation}
 Thus, we obtain by the Young inequality and the definition of $\vec{u}$
 that,
 \begin{equation}
  \begin{split}
   &
   \left|
   -\frac{1}{\pi(x,t)}
   \nabla^2
   \left(
   D(x)\log f
   +
   \phi(x)
   \right)
   \right|^{2}
   \\
   &\leq
   2
   |\nabla\vec{u}|^2
   +
   2
   \left|
   \frac{1}{\pi^2(x,t)}
   \nabla\pi(x,t)
   \otimes
   \nabla
   \left(
   D(x)\log f
   +
   \phi(x)
   \right)
   \right|^2
   \\
   &=
   2
   |\nabla\vec{u}|^2
   +
   2
   \left|
   \frac{1}{\pi(x,t)}
   \nabla\pi(x,t)
   \otimes
   \vec{u}
   \right|^2
   \\
   &=
   2
   |\nabla\vec{u}|^2
   +
   \frac{2}{\pi^2(x,t)}
   \left|
   \nabla\pi(x,t)
   \right|^2
   \left|
   \vec{u}
   \right|^2.
  \end{split}
 \end{equation}
 Therefore, we arrive at
 \begin{equation*}
 \begin{split}
&  \left(
   \int_\Omega
   |\vec{u}|^{p^*}f\,dx
	 \right)^{\frac{1}{p^*}}
    \\
    &\quad
  \leq
  \frac{\Cr{const:Pi_Max}\Cr{const:3.Sobolev}}{\Cr{const:Pi_Min}}
  \left(
   \int_\Omega
   \left(
   2
   |\nabla\vec{u}|^2
   +
   \frac{2}{\pi^2(x,t)}
   \left|
   \nabla\pi(x,t)
   \right|^2
   \left|
   \vec{u}
   \right|^2
   \right)
   f\,dx
  \right)^{\frac{1}{2}}.
  \end{split}
 \end{equation*}
 Finally, we compute the coefficient of $|\vec{u}|^2$ in the right hand
 side. Using \eqref{eq:4.Sobolev_Assumption}, the positive bounds for the mobility \eqref{eq:1.PiMinMax}, and the bound for the derivatives of the mobility \eqref{eq:1.grad_Pi}, we obtain
 \begin{equation}
  \frac{1}{\pi^2(x,t)}
   \left|
    \nabla\pi(x,t)
   \right|^2
   \leq
   \frac{\Cr{const:grad_Pi}^2}{\Cr{const:Pi_Min}^2}
   \leq
   \frac{\varepsilon}{2},
 \end{equation}
 hence we obtain \eqref{eq:4.SobolevType} by taking
 $\Cr{const:4.Sobolev}=\frac{\Cr{const:Pi_Max}\Cr{const:3.Sobolev}}{\Cr{const:Pi_Min}}$.
\end{proof}


\begin{remark}
 We emphasize that the  Sobolev-type constant $\Cr{const:4.Sobolev}$ in the above lemma is independent of
 $\Cr{const:grad_Pi}$, hence the gradient of the mobility $\pi$.
\end{remark}

With this result, next we will show that if the  bounds for the gradient mobility
$\Cr{const:grad_Pi}$, defined in \eqref{eq:1.grad_Pi}, is not too large, 
then we can obtain the Poincare type inequality for $\vec{u}$.

\begin{lemma}
 \label{lem:4.Poincare}
 Let $f$ be a solution of \eqref{eq:4.FokkerPlanck}, and let $\vec{u}$
 be given as in \eqref{eq:4.FokkerPlanck}. 
 Assume  the bound for the gradient mobility
$\Cr{const:grad_Pi}$, defined in \eqref{eq:1.grad_Pi} satisfies the following condition:
  \begin{equation}
   \label{eq:4.Poincare_Assumption}
    \Cr{const:grad_Pi}
    \leq
    \frac{\Cr{const:Pi_Min}}{2\Cr{const:4.Sobolev}},
  \end{equation}
 where $\Cr{const:Pi_Min}$ is the lower bound of $\pi$ appeared in \eqref{eq:1.PiMinMax}, $\Cr{const:4.Sobolev}$ is appeared in
 \eqref{eq:4.SobolevType}.
 Then, there is a constant
 $\Cl{const:4.Poincare}>0$ 
 which depends
 only on $n$, $\pi$ (in terms of $\Cr{const:Pi_Min}$,
 $\Cr{const:Pi_Max}$ defined in \eqref{eq:1.PiMinMax}), $f_0$ (in terms
 of $\Cr{const:InitMin}$, $\Cr{const:InitMax}$ defined in
 \eqref{eq:1.Initial}), the gradient of $D(x)$ (in terms of
 $\Cr{const:grad_D}$ defined in \eqref{eq:1.grad_D}), and
 $\|\phi\|_{L^\infty(\Omega)}$ such that:
 \begin{equation}
  \label{eq:4.Poincare}
   \int_\Omega|\vec{u}|^2f\,dx
   \leq
   \Cr{const:4.Poincare}
   \int_\Omega
   |\nabla\vec{u}|^{2}
   f\,dx.
 \end{equation}
\end{lemma}

\begin{remark}
 Note from the proof of Lemma \ref{lem:4.Sobolev}, that
 $\Cr{const:4.Sobolev}$ is taken as
 $\Cr{const:4.Sobolev}=\frac{\Cr{const:Pi_Max}\Cr{const:3.Sobolev}}{\Cr{const:Pi_Min}}$. Thus,
 the condition \eqref{eq:4.Poincare_Assumption} can be written as
 $\Cr{const:grad_Pi} \leq
 \frac{\Cr{const:Pi_Min}^2}{2\Cr{const:Pi_Max}\Cr{const:3.Sobolev}}$.
\end{remark}

\begin{proof}
 By the H\"older inequality and \eqref{eq:1.Mass_Conservation}, we have
\begin{equation}
 \left(
  \int_\Omega|\vec{u}|^2f\,dx
 \right)^{\frac{1}{2}}
 \leq
 \left(
  \int_\Omega|\vec{u}|^{p^*}f\,dx
 \right)^{\frac{1}{p^*}}.
\end{equation}
 We choose $\varepsilon>0$ later and assume
 \eqref{eq:4.Sobolev_Assumption}. Then, by the Sobolev type inequality
 \eqref{eq:4.SobolevType}, we obtain
 \begin{equation}
  \left(
   \int_\Omega|\vec{u}|^{p^*}f\,dx
  \right)^{\frac{1}{p^*}}
  \leq
  \Cr{const:4.Sobolev}
  \left(
   \int_\Omega
   \left(
   2|\nabla\vec{u}|^{2}
   +
   \varepsilon
   |\vec{u}|^{2}
   \right)
   f\,dx
  \right)^{\frac{1}{2}}.
 \end{equation}
 Now we choose $\varepsilon$ as
 \begin{equation}
  \label{eq:4.Poincare_epsilon}
  \Cr{const:4.Sobolev}^2
  \varepsilon
  =
  \frac{1}{2},
\end{equation}
and take $p^{*} = 2.$
 Then, we obtain
 \begin{equation}
  \frac{1}{2}
   \int_\Omega|\vec{u}|^2f\,dx
   \leq
   2\Cr{const:4.Sobolev}^2
   \int_\Omega
   |\nabla\vec{u}|^{2}
   f\,dx.
 \end{equation}
 Plugging \eqref{eq:4.Poincare_epsilon} into
 \eqref{eq:4.Sobolev_Assumption}, we get
 \eqref{eq:4.Poincare_Assumption}. Taking
 $\Cr{const:4.Poincare}=4\Cr{const:4.Sobolev}^2$, we obtain
 \eqref{eq:4.Poincare}.
\end{proof}


In contrast to Lemma \ref{lem:3.Sobolev}, the Sobolev-type of 
inequality \eqref{eq:4.SobolevType} has to include  an extra quadratic term of $\vec{u}$ in the
right hand side. Here we will  re-derive the  interpolation
inequality like that in  Lemma \ref{lem:3.Sobolev.cubic}.

\begin{lemma}
 \label{lem:4.Sobolev.cubic}
 Let $n=1,2,3$ and
 let
 $\varepsilon>0$, and let $\Cr{const:Pi_Min}, \Cr{const:grad_Pi}>0$ be
 the constants defined in the bounds for the mobility
 \eqref{eq:1.PiMinMax} and the bound of the derivative of the mobility
 \eqref{eq:1.grad_Pi} respectively. 
 Let $f$ be a solution of \eqref{eq:4.FokkerPlanck}, and let $\vec{u}$
 be given as in \eqref{eq:4.FokkerPlanck}. 
 Under the condition:
 \begin{equation}
  \label{eq:4.Sobolev.cubic.assumption}
  \Cr{const:grad_Pi}\leq\Cr{const:Pi_Min},
 \end{equation}
 we will have the following estimate:
 \begin{equation}
  \label{eq:4.Sobolev.cubic}
  \begin{split}
   \int_\Omega|\vec{u}|^3f\,dx
  & \leq
   \frac{3}{2}
   \Cr{const:4.Sobolev}^{\frac{3}{2}}
   \int_\Omega|\nabla \vec{u}|^2f\,dx
   +
   \frac{3}{2}
   \Cr{const:4.Sobolev}^{\frac{3}{2}}
   \int_\Omega|\vec{u}|^2f\,dx
   \\
&\quad
   +
   \frac{1}{4}
   \Cr{const:4.Sobolev}^{\frac{3}{2}}
   \left(
   \int_\Omega|\vec{u}|^2f\,dx
   \right)^3,
   \end{split}
 \end{equation}
 where $\Cr{const:4.Sobolev}$ is the constant defined in the Sobolev estimate 
 \eqref{eq:4.SobolevType} in Lemma \ref{lem:4.Sobolev}.
\end{lemma}

\begin{proof}
 We consider the case $n=3$ first. Let $\alpha,\beta>0$ be constants
 satisfying $\alpha+\beta=1$, and let $q>1$ be an exponent. Then, by the
 H\"older's inequality,
 \begin{equation}
  \label{eq:4.Sobolev.cubic.assumption1}
   \int_\Omega
   |\vec{u}|^3f
   \,dx
   \leq
   \left(
    \int_\Omega
   |\vec{u}|^{3\alpha q}f
   \,dx
   \right)^{\frac{1}{q}}
   \left(
    \int_\Omega
   |\vec{u}|^{3\beta q'}f
   \,dx
   \right)^{\frac{1}{q'}},
 \end{equation}
 where $q'$ is the H\"older's conjugate, namely, $\frac1q+\frac1{q'}=1$.
 Next, we put a constraint, $3\alpha q=p^\ast$, in order to apply the
 Sobolev type inequality \eqref{eq:4.SobolevType} with $\varepsilon=2$.
 Note that \eqref{eq:4.Sobolev.cubic.assumption} guarantees the
 assumption \eqref{eq:4.Sobolev_Assumption} in Lemma
 \ref{lem:4.Sobolev}. Then \eqref{eq:4.Sobolev.cubic.assumption1} turns
 into
 \begin{equation}
   \int_\Omega
   |\vec{u}|^3f
   \,dx
   \leq
    \Cr{const:4.Sobolev}^{\frac{p^\ast}{q}}
   \left(
    \int_\Omega
    2(|\nabla\vec{u}|^2+|\vec{u}|^2)
    f
    \,dx
   \right)^{\frac{p^\ast}{2q}}
   \left(
    \int_\Omega
    |\vec{u}|^{3\beta q'}f
   \,dx
   \right)^{\frac{1}{q'}}.
 \end{equation}
 Next, we set other constraints, $3\beta q'=2$ and
 $\frac{p^\ast}{2q}<1$. The Young's inequality implies,
 \begin{equation}
  \label{eq:4.Sobolev.cubic2}
   \begin{split}
    &
    \left(
    \int_\Omega
    2(|\nabla\vec{u}|^2+|\vec{u}|^2)
    f
    \,dx
    \right)^{\frac{p^\ast}{2q}}
    \left(
    \int_\Omega
    |\vec{u}|^{3\beta q'}f
    \,dx
    \right)^{\frac{1}{q'}}
    \\
    &\quad
    \leq
    \frac{p^\ast}{q}
    \int_\Omega
    (|\nabla\vec{u}|^2+|\vec{u}|^2)
    f
    \,dx
   +
    \left(
    1-
    \frac{p^\ast}{2q}
    \right)
    \left(
    \int_\Omega
    |\vec{u}|^{2}f
    \,dx
    \right)^{\frac{1}{q'}
    \left(
    1-
    \frac{p^\ast}{2q}
    \right)^{-1}
    },
   \end{split}
 \end{equation}
which yields:
 \begin{equation}
  \label{eq:4.Sobolev.cubic3}
  \begin{split}
   \int_\Omega
   |\vec{u}|^3f
   \,dx
 &  \leq
   \Cr{const:4.Sobolev}^{\frac{p^\ast}{q}}
   \frac{p^*}{q}
  \int_\Omega
    (|\nabla\vec{u}|^2+|\vec{u}|^2)
    f
    \,dx
    \\
    &\qquad
   +
   \Cr{const:4.Sobolev}^{\frac{p^\ast}{q}}
   \left(
   1-
   \frac{p^\ast}{2q}
   \right)
   \left(
   \int_\Omega
   |\vec{u}|^{2}f
   \,dx
   \right)^{\frac{1}{q'}
   \left(
   1-
    \frac{p^\ast}{2q}
   \right)^{-1}
   }
  .
  \end{split}
 \end{equation}

 Now we can come back to  examine the constraints. Since $3\alpha q=p^\ast$, $3\beta q'=2$, $\alpha+\beta=1$, and the properties of $q'$, $p^\ast$ imply that,
$  \frac{3\beta}{2}=\frac{1}{q'}=1-\frac{3\alpha}{p^\ast}=1-\frac{3\alpha}{2}+\frac{3\alpha}{n}.$
From these, one can deduce that  $\alpha=\frac{n}{6}=\frac{1}{2}$, and in turns, 
 $\beta=\frac{1}{2}$, $q=4$ and $p^*=6$. Plugging these values in \eqref{eq:4.Sobolev.cubic3} and
 we can obtain 
 \eqref{eq:4.Sobolev.cubic} directly.

 For the case $n=1,2$, we can take $p^\ast=6$, the same as in the case
 $n=3$. Then it is easy to verify that by taking $\alpha=\beta=\frac{1}{2}$ and $q=4$ in
 \eqref{eq:4.Sobolev.cubic3}, one can obtain the inequality
 \eqref{eq:4.Sobolev.cubic}.
\end{proof}

Next we  recall the following higher order energy law of  \eqref{eq:4.FokkerPlanck}, which can be derived by
direct computations as we had done in  \cite{MR4506846}. 

\begin{lemma}%
 [{\cite[Proposition 4.15]{MR4506846}}]
 Let $f$ be a solution of \eqref{eq:4.FokkerPlanck}, and let $\vec{u}$ be given as
 in \eqref{eq:4.FokkerPlanck}. Then, we have the following energy law,
 \begin{equation}
  \label{eq:4.Second_Derivative_Energy_Law}
   \begin{split}
    \frac{d^2}{dt^2}F[f](t)
    &=
    2\int_\Omega
    ((\nabla^2\phi(x))\vec{u}\cdot\vec{u})f\,dx
    +
    2\int_\Omega D(x)|\nabla \vec{u}|^2f\,dx
    \\
    &\qquad
    -
    \int_\Omega
    (\log f-1)
    \nabla |\vec{u}|^2\cdot\nabla D(x) f
    \,dx
   \\
   &\qquad
    -
    2
    \int_\Omega
    (1+\log f)
    \vec{u}\cdot\nabla D(x)\Div\vec{u} f\,dx
    \\
    &\qquad
    +
    2
    \int_\Omega
    \frac{\pi(x,t)}{D(x)}|\vec{u}|^2
    \log f\left(\vec{u}\cdot\nabla D(x)\right) f\,dx
    \\
    &\qquad
    +
    2
     \int_\Omega
    \frac{1}{D(x)}
    (\log f)^2\left(\vec{u}\cdot\nabla D(x)\right)^2 f\,dx
    \\
    &\qquad
    +
    2\int_\Omega
   \frac{1}{D(x)}
    \log f\left(\vec{u}\cdot\nabla D(x)\right)\left(\vec{u}\cdot\nabla \phi(x)\right) f\,dx
    \\
    &\qquad
    +
    \int_\Omega\pi_t(x,t)|\vec{u}|^2f\,dx
    +
    \int_\Omega
    |\vec{u}|^2\vec{u}\cdot\nabla \pi(x,t) f
    \,dx
    \\
    &\qquad
    -2
    \int_\Omega
    (\log f-1)
    \frac{1}{\pi(x,t)}
    |\vec{u}|^2\nabla\pi(x,t)\cdot \nabla D(x)
    f\,dx
    \\
    &\qquad
    +2
    \int_\Omega
    (\log f-1)
    \frac{1}{\pi(x,t)}
    (\vec{u}\cdot\nabla\pi(x,t))(\vec{u}\cdot\nabla D(x))
    f\,dx
    \\
    &\qquad
    +
    \int_\Omega
    \frac{D(x)}{\pi(x,t)}((\nabla|\vec{u}|^2)\cdot\nabla\pi(x,t))
    f\,dx
    \\
    &\qquad
    -
    2
    \int_\Omega
    \frac{D(x)}{\pi(x,t)}((\nabla\vec{u})\vec{u}\cdot\nabla\pi(x,t))
    f\,dx.
   \end{split}
 \end{equation}
\end{lemma}

We will first derive the following inequality from the second derivative of
the free energy above \eqref{eq:4.Second_Derivative_Energy_Law}.

\begin{lemma}
 Let $f$ be a solution of \eqref{eq:4.FokkerPlanck}, and let $\vec{u}$ be given as
 in \eqref{eq:4.FokkerPlanck}. Then, 
 we have
 \begin{equation}
  \label{eq:4.Second_Derivative_Free_Energy.1}
  \begin{split}
   &
   \frac{d^2}{dt^2}F[f](t)
    \\
   &\geq
   2\int_\Omega
   \biggl(
   (\nabla^2\phi(x)\vec{u}\cdot\vec{u})
   +
   \frac{1}{2}\pi_t(x,t)|\vec{u}|^2
   \\
   &\qquad
   -
   \Bigl(
   \frac{1}{2}
   +
   \frac{2(\Cr{const:log_f}+1)\Cr{const:grad_D}\Cr{const:grad_Pi}}{\Cr{const:Pi_Min}}
   +
   \frac{\Cr{const:log_f}\Cr{const:grad_D}}{\Cr{const:D_Min}}
   \|\nabla \phi\|_{L^\infty(\Omega)}
   \Bigr) |\vec{u}|^2
   \biggr) f\,dx
   \\
   &\quad
   +
   2\int_\Omega
   \biggl(
   1
   -
   \frac{2(1+n)\Cr{const:grad_D}^2(\Cr{const:log_f}+1)^2}{\Cr{const:D_Min}}
   -
   \frac{4D(x)\Cr{const:grad_Pi}^2}{\Cr{const:Pi_Min}^2}
   \biggr)
   D(x)|\nabla \vec{u}|^2 f\,dx
   \\
   &\quad
   +
   2
   \int_\Omega
   \frac{\pi(x,t)}{D(x)}|\vec{u}|^2
   \log f\left(\vec{u}\cdot\nabla D(x)\right) f\,dx
    \\
   &\quad
   +
   \int_\Omega
   |\vec{u}|^2\vec{u}\cdot\nabla \pi(x,t) f
   \,dx,
  \end{split}
 \end{equation}
 where $\Cr{const:D_Min}$ is the lower bound of $D(x)$
 defined in \eqref{eq:1.D_Min}, $\Cr{const:Pi_Min}$ is the lower bound
 of $\pi(x,t)$ defined in \eqref{eq:1.PiMinMax}, $\Cr{const:grad_Pi}$ is
 the upper bound for the estimate of the gradient of $\pi(x,t)$ defined
 in \eqref{eq:1.grad_Pi}, $\Cr{const:grad_D}$ is the upper bound for
 the estimate of the gradient of $D(x)$ defined in \eqref{eq:1.grad_D},
 and $\Cr{const:log_f}$ is the bound for the estimate of $\log f$
 defined in \eqref{eq:1.bounds_of_log_f}.
\end{lemma}

\begin{proof}
We start with  the estimates of the integrands for the 3rd, 4th, and 7th
 integrals of \eqref{eq:4.Second_Derivative_Energy_Law}.  
 By Cauchy's inequality with $\varepsilon=1/8$ (See
 \cite[p.662]{MR1625845}) and the definition of $\Cr{const:D_Min}$, we
 will get
 \begin{equation}
  \begin{split}
   |
   (\log f-1)\nabla |\vec{u}|^2\cdot\nabla D(x) f
   |
   &\leq
   4
   (|\log f|+1)^2
   |\nabla \vec{u}|^2
   |\nabla D(x)|^2
   f
   +
   \frac{1}{4}
   |\vec{u}|^2 f
   \\
   &\leq
   \frac{4\Cr{const:grad_D}^2(\Cr{const:log_f}+1)^2}{\Cr{const:D_Min}}
   D(x)
   |\nabla \vec{u}|^2
   f
   +
   \frac{1}{4}
   |\vec{u}|^2 f,
  \end{split}
 \end{equation}
 \begin{equation}
  \begin{split}
   |2(1+\log f)\vec{u}\cdot\nabla D(x)\Div\vec{u}f|
   &\leq
   4n
   (|\log f|+1)^2
   |\nabla \vec{u}|^2
   |\nabla D(x)|^2 f
   +
   \frac{1}{4}
   |\vec{u}|^2 f
   \\
   &\leq
   \frac{4n\Cr{const:grad_D}^2(\Cr{const:log_f}+1)^2}{\Cr{const:D_Min}}
   D(x)
   |\nabla \vec{u}|^2
   f
   +
   \frac{1}{4}
   |\vec{u}|^2 f,
  \end{split}
 \end{equation}
 and
 \begin{equation}
  \left|
   \frac{2}{D(x)}
   \log f\left(\vec{u}\cdot\nabla D(x)\right)\left(\vec{u}\cdot\nabla \phi(x)\right) f
  \right|
  \leq
  \frac{2\Cr{const:log_f}\Cr{const:grad_D}}{\Cr{const:D_Min}}
  \|\nabla \phi\|_{L^\infty(\Omega)}
  |\vec{u}|^2f.
 \end{equation}
Next we  estimate the 10th, 11st, 12nd and 13rd terms of the
 right-hand side of \eqref{eq:4.Second_Derivative_Energy_Law}. Using the
 Cauchy-Schwarz inequality that
 \begin{equation}
  \left|
   2(\log f-1)
   \frac{1}{\pi(x,t)}
   |\vec{u}|^2\nabla\pi(x,t)\cdot \nabla D(x)
   f
  \right|
  \leq
  \frac{2(\Cr{const:log_f}+1)\Cr{const:grad_D}\Cr{const:grad_Pi}}{\Cr{const:Pi_Min}}
  |\vec{u}|^2
  f,
 \end{equation}
 \begin{equation}
  \left|
   2(\log f-1)
   \frac{1}{\pi(x,t)}
   (\vec{u}\cdot\nabla\pi(x,t))(\vec{u}\cdot\nabla D(x))
   f
  \right|
  \leq
  \frac{2(\Cr{const:log_f}+1)\Cr{const:grad_D}\Cr{const:grad_Pi}}{\Cr{const:Pi_Min}}
  |\vec{u}|^2
  f,
 \end{equation}
 \begin{equation}
    \begin{split}
  \left|
   \frac{D(x)}{\pi(x,t)}((\nabla|\vec{u}|^2)\cdot\nabla\pi(x,t))f
  \right|
  &\leq
   \frac{2D(x)}{\pi(x,t)}
   |\nabla \vec{u}||\vec{u}||\nabla \pi(x,t)| f
   \\
   &\leq
   \frac{4(D(x))^2\Cr{const:grad_Pi}^2}{\Cr{const:Pi_Min}^2}
   |\nabla \vec{u}|^2 f
   +
   \frac{1}{4}|\vec{u}|^2f,
    \end{split}
 \end{equation}
 and
 \begin{equation}
    \begin{split}
  \left|
   2
   \frac{D(x)}{\pi(x,t)}((\nabla\vec{u})\vec{u}\cdot\nabla\pi(x,t))f
  \right|
  &\leq
  \frac{2D(x)}{\pi(x,t)}|\nabla \vec{u}||\vec{u}||\nabla \pi(x,t)| f
    \\
    &\leq
  \frac{4(D(x))^2\Cr{const:grad_Pi}^2}{\Cr{const:Pi_Min}^2}
  |\nabla \vec{u}|^2 f
  +
  \frac{1}{4}|\vec{u}|^2f.
    \end{split}
 \end{equation}
 Combining all inequalities above, ignoring the
 integral of $\frac{1}{D(x)}(\log f)^2(\vec{u}\cdot\nabla D(x))^2f$,
 and use them in \eqref{eq:4.Second_Derivative_Energy_Law}, we arrive
 at the reduced estimate \eqref{eq:4.Second_Derivative_Free_Energy.1}.
\end{proof}

Next we consider the terms involving cubic order of $\vec{u}$ in
\eqref{eq:4.Second_Derivative_Free_Energy.1}.

\begin{lemma}
 Let $f$ be a solution of \eqref{eq:4.FokkerPlanck}, and let $\vec{u}$ be given as
 in \eqref{eq:4.FokkerPlanck}. Assume \eqref{eq:4.Sobolev.cubic.assumption} and \eqref{eq:1.D_Min}.
 Then, we have,
 \begin{equation}
  \label{eq:4.Second_Derivative_Free_Energy.2}
   \begin{split}
    &
    2
    \left|
    \int_\Omega
    \frac{\pi(x,t)}{D(x)}|\vec{u}|^2
    \log f\left(\vec{u}\cdot\nabla D(x)\right) f\,dx
    \right|
    +
    \left|
    \int_\Omega
    |\vec{u}|^2\vec{u}\cdot\nabla \pi(x,t) f
    \,dx
    \right|
    \\
    &\leq
    \left(
    \frac{3\Cr{const:log_f}\Cr{const:Pi_Max}\Cr{const:grad_D}}{\Cr{const:D_Min}}
    +
   \frac{3\Cr{const:grad_Pi}}{2}
    \right)
\Cr{const:4.Sobolev}^{\frac{3}{2}}    \int_\Omega D(x)|\nabla\vec{u}|^2f\,dx
    \\
    &\qquad
    +
    \left(
    \frac{3\Cr{const:log_f}\Cr{const:Pi_Max}\Cr{const:grad_D}}{\Cr{const:D_Min}}
   +
    \frac{3\Cr{const:grad_Pi}}{2}
    \right)
\Cr{const:4.Sobolev}^{\frac{3}{2}}    \int_\Omega |\vec{u}|^2f\,dx
   \\
    &\qquad
    +
    \left(
    \frac{\Cr{const:log_f}\Cr{const:Pi_Max}\Cr{const:grad_D}}{2\Cr{const:D_Min}}
    +
    \frac{\Cr{const:grad_Pi}}{4}
\right)
\Cr{const:4.Sobolev}^{\frac{3}{2}}   \left(
   \int_\Omega |\vec{u}|^2f\,dx
   \right)^3,
   \end{split}
 \end{equation}
 where $\Cr{const:D_Min}$ is the lower bound of $D(x)$
 defined in \eqref{eq:1.D_Min}, $\Cr{const:Pi_Max}$ is the upper bound
 of $\pi(x,t)$ defined in \eqref{eq:1.PiMinMax}, $\Cr{const:grad_Pi}$ is
 the upper bound for the estimate of the gradient of $\pi(x,t)$ defined
 in \eqref{eq:1.grad_Pi}, $\Cr{const:grad_D}$ is the upper bound for
 the estimate of the gradient of $D(x)$ defined in \eqref{eq:1.grad_D},
 $\Cr{const:log_f}$ is the bound for the estimate of $\log f$ defined in
 \eqref{eq:1.bounds_of_log_f}, and $\Cr{const:4.Sobolev}$ is appeared in
 \eqref{eq:4.SobolevType}.
\end{lemma}

\begin{proof}
 First, we compute the first term of the left hand side of
 \eqref{eq:4.Second_Derivative_Free_Energy.2}.
 Using
 \eqref{eq:4.Sobolev.cubic} and 
 \eqref{eq:1.D_Min},
 we compute
 \begin{equation}
  \label{eq:4.Second_Derivative_Free_Energy.2-1}
   \begin{split}
    &
    \left|
    2
    \int_\Omega
    \frac{\pi(x,t)}{D(x)}|\vec{u}|^2
    \log f\left(\vec{u}\cdot\nabla D(x)\right) f\,dx
    \right|
 \leq
    \frac{2\Cr{const:log_f}\Cr{const:Pi_Max}\Cr{const:grad_D}}{\Cr{const:D_Min}}
    \int_\Omega|\vec{u}|^3f\,dx
    \\
    &\leq
    \frac{3\Cr{const:log_f}\Cr{const:Pi_Max}\Cr{const:grad_D}\Cr{const:4.Sobolev}^{\frac{3}{2}}}{\Cr{const:D_Min}}
   \int_\Omega|\nabla\vec{u}|^2f\,dx
   \\
   &\qquad
    +
    \frac{3\Cr{const:log_f}\Cr{const:Pi_Max}\Cr{const:grad_D}\Cr{const:4.Sobolev}^{\frac{3}{2}}}{\Cr{const:D_Min}}
    \int_\Omega|\vec{u}|^2f\,dx
   +
    \frac{\Cr{const:log_f}\Cr{const:Pi_Max}\Cr{const:grad_D}\Cr{const:4.Sobolev}^{\frac{3}{2}}}{2\Cr{const:D_Min}}
   \left(
   \int_\Omega |\vec{u}|^2f\,dx
   \right)^3
    \\
    &\leq
    \frac{3\Cr{const:log_f}\Cr{const:Pi_Max}\Cr{const:grad_D}\Cr{const:4.Sobolev}^{\frac{3}{2}}}{\Cr{const:D_Min}}
    \int_\Omega D(x)|\nabla\vec{u}|^2f\,dx
   \\
   &\qquad
    +
    \frac{3\Cr{const:log_f}\Cr{const:Pi_Max}\Cr{const:grad_D}\Cr{const:4.Sobolev}^{\frac{3}{2}}}{\Cr{const:D_Min}}
    \int_\Omega |\vec{u}|^2f\,dx
   +
    \frac{\Cr{const:log_f}\Cr{const:Pi_Max}\Cr{const:grad_D}\Cr{const:4.Sobolev}^{\frac{3}{2}}}{2\Cr{const:D_Min}}
   \left(
   \int_\Omega |\vec{u}|^2f\,dx
   \right)^3.
   \end{split}
 \end{equation}
 Next, we consider the second term of the left hand side of
 \eqref{eq:4.Second_Derivative_Free_Energy.2}.  Using
 \eqref{eq:4.Sobolev.cubic} and 
 \eqref{eq:1.D_Min} again that,
 \begin{equation}
  \label{eq:4.Second_Derivative_Free_Energy.2-2}
  \begin{split}
   &
   \left|
   \int_\Omega
   |\vec{u}|^2\vec{u}\cdot\nabla\pi(x,t)
   f\,dx
   \right|
   \leq
   \Cr{const:grad_Pi}
   \int_\Omega|\vec{u}|^3f\,dx
   \\
   &\leq
   \frac{3\Cr{const:grad_Pi}\Cr{const:4.Sobolev}^{\frac{3}{2}}}{2}
   \int_\Omega|\nabla \vec{u}|^2f\,dx
   +
   \frac{3\Cr{const:grad_Pi}\Cr{const:4.Sobolev}^{\frac{3}{2}}}{2}
   \int_\Omega|\vec{u}|^2f\,dx
   +
   \frac{\Cr{const:grad_Pi}\Cr{const:4.Sobolev}^{\frac{3}{2}}}{4}
   \left(
   \int_\Omega|\vec{u}|^2f\,dx
   \right)^3
   \\
   &\leq
   \frac{3\Cr{const:grad_Pi}\Cr{const:4.Sobolev}^{\frac{3}{2}}}{2}
   \int_\Omega D(x)|\nabla \vec{u}|^2f\,dx
   +
   \frac{3\Cr{const:grad_Pi}\Cr{const:4.Sobolev}^{\frac{3}{2}}}{2}
   \int_\Omega|\vec{u}|^2f\,dx
   \\
   &\qquad
   +
   \frac{\Cr{const:grad_Pi}\Cr{const:4.Sobolev}^{\frac{3}{2}}}{4}
   \left(
   \int_\Omega|\vec{u}|^2f\,dx
   \right)^3.
  \end{split}
 \end{equation}
 Combining \eqref{eq:4.Second_Derivative_Free_Energy.2-1} and
 \eqref{eq:4.Second_Derivative_Free_Energy.2-2}, we obtain
 \eqref{eq:4.Second_Derivative_Free_Energy.2}.
\end{proof}

Now we are ready to demonstrate that under the assumptions 
on suitable sufficient ``smallness'' conditions for the terms involving  
the gradient of the mobility $|\nabla\pi|$, and $|\pi_t|$,  in terms of $\Cr{const:grad_Pi}$ and $\Cr{const:Pi_Time}$,
together with the assumption of $\Cr{const:D_Min}$ being sufficiently
large (the lower bound of the diffusion), 
one can deduce from \eqref{eq:4.Second_Derivative_Energy_Law} the following proposition.

\begin{proposition}
 \label{prop:4.Second_Derivative_Free_Energy} 
 Let $f$ be a smooth classical solution of
 \eqref{eq:4.FokkerPlanck}, and let $\vec{u}$ be given as in \eqref{eq:4.FokkerPlanck}. Assume
 the lower bound of the diffusion $D(x)$, $\Cr{const:D_Min}\geq1$, is
 large enough and the upper bounds of $|\pi_t|$ and $|\nabla\pi|$, in
 terms of $\Cr{const:Pi_Time}\leq\frac{1}{6}$ and $\Cr{const:grad_Pi}$, are small
 enough such that the following conditions hold:
 \begin{equation}
  \label{eq:4.Second_Derivative_Free_Energy.3.AssumptionA}
  \begin{split}
   \max
   \biggl\{
  & 12\Cr{const:log_f}\Cr{const:Pi_Max}\Cr{const:grad_D}\Cr{const:4.Sobolev}^{\frac{3}{2}},
   12\Cr{const:log_f}\Cr{const:grad_D}\|\nabla\phi\|_{L^\infty(\Omega)},
   \\
 &  16(1+n)(\Cr{const:log_f}+1)^2\Cr{const:grad_D}^2
   \biggr\}
   \leq
   \Cr{const:D_Min},
   \end{split}
 \end{equation}
 and
 \begin{equation}
  \label{eq:4.Second_Derivative_Free_Energy.3.AssumptionC}
   \Cr{const:grad_Pi}
   \leq
   \min
   \left\{
   \frac{1}{6\Cr{const:4.Sobolev}^{\frac{3}{2}}},
   \frac{\Cr{const:Pi_Min}}{24(\Cr{const:log_f}+1)\Cr{const:grad_D}},
   \frac{\Cr{const:Pi_Min}}{4\sqrt{2(\sqrt{n}\Cr{const:grad_D}+1)\Cr{const:D_Min}}
   },
   \Cr{const:Pi_Min}
   \right\}.
 \end{equation}
 Here, $\Cr{const:D_Min}$ is the lower bound of $D(x)$
 defined in \eqref{eq:1.D_Min}, $\Cr{const:Pi_Min}$ and
 $\Cr{const:Pi_Max}$ are the lower and the upper bounds of $\pi(x,t)$
 defined in \eqref{eq:1.PiMinMax} respectively. The bound $\Cr{const:Pi_Time}$ is
 the bound for the estimate of the time derivative of $\pi$ defined in
 \eqref{eq:1.Pi_Time}, $\Cr{const:grad_Pi}$ is the upper bound for the
 estimate of the gradient of $\pi(x,t)$ defined in \eqref{eq:1.grad_Pi},
 $\Cr{const:grad_D}$ is the upper bound for the estimate of the gradient
 of $D(x)$ defined in \eqref{eq:1.grad_D}, $\Cr{const:log_f}$ is the
 bound for the estimate of $\log f$ defined in
 \eqref{eq:1.bounds_of_log_f}, and $\Cr{const:4.Sobolev}$ is appeared in
 \eqref{eq:4.SobolevType}.
 Then, we obtain,
 \begin{equation}
  \label{eq:4.Second_Derivative_Free_Energy.3}
  \begin{split}
   \frac{d^2}{dt^2}F[f](t)
&   \geq
   -2(\lambda+1)
   \int_\Omega
   |\vec{u}|^2
   f
   \,dx
   +
   \int_\Omega
   D(x)|\nabla \vec{u}|^2 f\,dx
   \\
   &\quad
   -\frac{1}{12}
   \left(
   \int_\Omega
   |\vec{u}|^2
   f
   \,dx
   \right)^3.
   \end{split}
 \end{equation}
\end{proposition}

\begin{remark}
 We emphasize that all of the constants in the left-hand side of
 \eqref{eq:4.Second_Derivative_Free_Energy.3.AssumptionA} are independent
 of $\Cr{const:D_Min}$ and $\Cr{const:grad_Pi}$, and every constant in
 the right-hand side of
 \eqref{eq:4.Second_Derivative_Free_Energy.3.AssumptionC} except for
 $\Cr{const:D_Min}$ is independent of $\Cr{const:grad_Pi}$. Thus, in
 order to guarantee
 \eqref{eq:4.Second_Derivative_Free_Energy.3.AssumptionA} and
 \eqref{eq:4.Second_Derivative_Free_Energy.3.AssumptionC}, first we fix
 sufficiently small $\Cr{const:D_Min}$ so 
 \eqref{eq:4.Second_Derivative_Free_Energy.3.AssumptionA} holds. Next, for the
 fixed $\Cr{const:D_Min}$, we take sufficiently large
 $\Cr{const:grad_Pi}$ to ensure the inequality
 \eqref{eq:4.Second_Derivative_Free_Energy.3.AssumptionC}.
\end{remark}

\begin{proof}
 Note that the conditions 
 \eqref{eq:4.Second_Derivative_Free_Energy.3.AssumptionA},
 $\Cr{const:Pi_Time}\leq\frac{1}{6}$
 , and
 \eqref{eq:4.Second_Derivative_Free_Energy.3.AssumptionC}
 yield the following inequalities:
 \begin{equation}
  \label{eq:4.Second_Derivative_Free_Energy.3.Assumption1}
   \left(
    \frac{3\Cr{const:log_f}\Cr{const:Pi_Max}\Cr{const:grad_D}}{\Cr{const:D_Min}}
    +
    \frac{3\Cr{const:grad_Pi}}{2}
   \right)
   \Cr{const:4.Sobolev}^{\frac{3}{2}}
   \leq
   \frac{1}{2},
 \end{equation}
 \begin{equation}
  \label{eq:4.Second_Derivative_Free_Energy.3.Assumption3}
   \frac{1}{2}\Cr{const:Pi_Time}
    +
    \frac{2(\Cr{const:log_f}+1)\Cr{const:grad_Pi}\Cr{const:grad_D}}{\Cr{const:Pi_Min}}
    +
    \frac{\Cr{const:log_f}\Cr{const:grad_D}}{\Cr{const:D_Min}}
    \|\nabla\phi\|_{L^\infty(\Omega)}
   \leq
   \frac{1}{4},
 \end{equation}
 \begin{equation}
  \label{eq:4.Second_Derivative_Free_Energy.3.Assumption4}
   (\Cr{const:log_f}+1)^2
   \Cr{const:grad_D}^2
   \leq
   \frac{1}{16(1+n)}\Cr{const:D_Min},
 \end{equation}
 \begin{equation}
  \label{eq:4.Second_Derivative_Free_Energy.3.Assumption5}
   (\sqrt{n}\Cr{const:grad_D}+1)
   \Cr{const:D_Min}
   \Cr{const:grad_Pi}^2
   \leq
   \frac{1}{32}\Cr{const:Pi_Min}^2,
 \end{equation}
 and
$  \Cr{const:grad_Pi}
   \leq
   \Cr{const:Pi_Min}.$

 First, combining \eqref{eq:4.Second_Derivative_Free_Energy.1} and
 \eqref{eq:4.Second_Derivative_Free_Energy.2}, we obtain
 \begin{equation}
  \label{eq:4.Second_Derivative_Free_Energy.3.1}
  \begin{split}
   \frac{d^2}{dt^2}F[f](t)
   &
   \geq
   2\int_\Omega
   \biggl(
   (\nabla^2\phi(x)\vec{u}\cdot\vec{u})
   +
   \frac{1}{2}\pi_t(x,t)|\vec{u}|^2
   \\
   &\quad
   -
   \Bigl(
   \frac{1}{2}
   +
   \frac{2(\Cr{const:log_f}+1)\Cr{const:grad_D}\Cr{const:grad_Pi}}{\Cr{const:Pi_Min}}
   +
   \frac{\Cr{const:log_f}\Cr{const:grad_D}}{\Cr{const:D_Min}}
   \|\nabla \phi\|_{L^\infty(\Omega)}
   \Bigr) |\vec{u}|^2
   \biggr) f\,dx
   \\
   &\quad
   +
   2\int_\Omega
   \biggl(
   1
   -
   \frac{2(1+n)\Cr{const:grad_D}^2(\Cr{const:log_f}+1)^2}{\Cr{const:D_Min}}
   \\
    &\qquad\qquad\qquad
   -
   \frac{4D(x)\Cr{const:grad_Pi}^2}{\Cr{const:Pi_Min}^2}
   \biggr)
   D(x)|\nabla \vec{u}|^2 f\,dx
   \\
   &\quad
   -
   \left(
   \frac{3\Cr{const:log_f}\Cr{const:Pi_Max}\Cr{const:grad_D}}{\Cr{const:D_Min}}
   +
   \frac{3\Cr{const:grad_Pi}}{2}
   \right)
   \Cr{const:4.Sobolev}^{\frac{3}{2}}
   \int_\Omega D(x)|\nabla\vec{u}|^2f\,dx
   \\
   &\quad
   -
   \left(
   \frac{3\Cr{const:log_f}\Cr{const:Pi_Max}\Cr{const:grad_D}}{\Cr{const:D_Min}}
   +
   \frac{3\Cr{const:grad_Pi}}{2}
   \right)
   \Cr{const:4.Sobolev}^{\frac{3}{2}}
   \int_\Omega |\vec{u}|^2f\,dx
   \\
   &\quad
   -
   \left(
   \frac{\Cr{const:log_f}\Cr{const:Pi_Max}\Cr{const:grad_D}}{2\Cr{const:D_Min}}
   +
   \frac{\Cr{const:grad_Pi}}{4}
   \right)
   \Cr{const:4.Sobolev}^{\frac{3}{2}}
   \left(
   \int_\Omega |\vec{u}|^2f\,dx
   \right)^3.
  \end{split}
 \end{equation}

 We consider the integral of $|\vec{u}|^2f$ in
 \eqref{eq:4.Second_Derivative_Free_Energy.3.1}.  Using
 \eqref{eq:1.Hesse_Potential_Lower_Bounds} and \eqref{eq:1.Pi_Time}, we
 get
 \begin{equation}
  (\nabla^2\phi(x)\vec{u}\cdot\vec{u})
   +
   \frac{1}{2}\pi_t(x,t)|\vec{u}|^2
   \geq
   \left(
    -
    \lambda
    -
    \frac{1}{2}\Cr{const:Pi_Time}
   \right)
   |\vec{u}|^2.
 \end{equation}
 Then applying the assumptions
 \eqref{eq:4.Second_Derivative_Free_Energy.3.Assumption1} and
 \eqref{eq:4.Second_Derivative_Free_Energy.3.Assumption3}, we have
 \begin{equation}
  \label{eq:4.Second_Derivative_Free_Energy.3.2}
  \begin{split}
   &
   2\int_\Omega
   \biggl(
   (\nabla^2\phi(x)\vec{u}\cdot\vec{u})
   +
   \frac{1}{2}\pi_t(x,t)|\vec{u}|^2
   \\
   &\qquad
   -
   \Bigl(
   \frac{1}{2}
   +
   \frac{2(\Cr{const:log_f}+1)\Cr{const:grad_D}\Cr{const:grad_Pi}}{\Cr{const:Pi_Min}}
   +
   \frac{\Cr{const:log_f}\Cr{const:grad_D}}{\Cr{const:D_Min}}
   \|\nabla \phi\|_{L^\infty(\Omega)}
   \Bigr) |\vec{u}|^2
   \biggr) f\,dx
   \\
   &\qquad
   -
   \left(
   \frac{3\Cr{const:log_f}\Cr{const:Pi_Max}\Cr{const:grad_D}}{\Cr{const:D_Min}}
   +
   \frac{3\Cr{const:grad_Pi}}{2}
   \right)
   \Cr{const:4.Sobolev}^{\frac{3}{2}}
   \int_\Omega |\vec{u}|^2f\,dx
   \\
   &\geq
   -2
   \biggl(
   \lambda
   +
   \frac{1}{2}\Cr{const:Pi_Time}
   +
   \frac{1}{2}
   +
   \frac{2(\Cr{const:log_f}+1)\Cr{const:grad_D}\Cr{const:grad_Pi}}{\Cr{const:Pi_Min}}
   +
   \frac{\Cr{const:log_f}\Cr{const:grad_D}}{\Cr{const:D_Min}}
   \|\nabla \phi\|_{L^\infty(\Omega)}
   \\
   &\qquad\quad
   +
   \frac{1}{2}
   \left(
   \frac{3\Cr{const:log_f}\Cr{const:Pi_Max}\Cr{const:grad_D}}{\Cr{const:D_Min}}
   +
   \frac{3\Cr{const:grad_Pi}}{2}
   \right)
   \Cr{const:4.Sobolev}^{\frac{3}{2}}  
   \biggr)
   \int_\Omega
   |\vec{u}|^2
   f
   \,dx
   \\
   &
   \geq
   -2(\lambda+1)
   \int_\Omega
   |\vec{u}|^2
   f
   \,dx.
  \end{split}
 \end{equation}
 Next we consider the integrals in \eqref{eq:4.Second_Derivative_Free_Energy.3.1} that contain term with $D(x)|\nabla\vec{u}|^2f$. Note, using
 \eqref{eq:4.Second_Derivative_Free_Energy.3.Assumption1}, we obtain
 \begin{equation}
  \label{eq:4.Second_Derivative_Free_Energy.3.3}
   \begin{split}
    &
    2
    \biggl(
    1
    -
    \frac{2(1+n)\Cr{const:grad_D}^2(\Cr{const:log_f}+1)^2}{\Cr{const:D_Min}}
    -
    \frac{4D(x)\Cr{const:grad_Pi}^2}{\Cr{const:Pi_Min}^2}
    \biggr)
    \\
    &\qquad
    -
    \left(
    \frac{3\Cr{const:log_f}\Cr{const:Pi_Max}\Cr{const:grad_D}}{\Cr{const:D_Min}}
   +
   \frac{3\Cr{const:grad_Pi}}{2}
   \right)
    \Cr{const:4.Sobolev}^{\frac{3}{2}}
   \\
   &\geq
   2
   \biggl(
   1
   -
    \frac{2(1+n)\Cr{const:grad_D}^2(\Cr{const:log_f}+1)^2}{\Cr{const:D_Min}}
    -
    \frac{4\Cr{const:grad_Pi}^2}{\Cr{const:Pi_Min}^2}
   \|D\|_{L^\infty(\Omega)}
   \biggr)
   -\frac{1}{2}.
   \end{split}
 \end{equation}
 Here we use \eqref{eq:1.D_Max} and $\Cr{const:D_Min}\geq1$, then
 \begin{equation}
  \|D\|_{L^\infty(\Omega)}
   \leq
   \Cr{const:D_Min}
   +
   \sqrt{n}\Cr{const:grad_D}
   \leq
   \left(
    1
    +
    \sqrt{n}\Cr{const:grad_D}
   \right)
   \Cr{const:D_Min}.
 \end{equation}
 Thus, by the assumptions of
 \eqref{eq:4.Second_Derivative_Free_Energy.3.Assumption4} and
 \eqref{eq:4.Second_Derivative_Free_Energy.3.Assumption5}, we have that
 \begin{equation}
  \label{eq:4.Second_Derivative_Free_Energy.3.4}
   \frac{2(1+n)\Cr{const:grad_D}^2(\Cr{const:log_f}+1)^2}{\Cr{const:D_Min}}
   +
   \frac{4\Cr{const:grad_Pi}^2}{\Cr{const:Pi_Min}^2}
   \|D\|_{L^\infty(\Omega)}
   \leq
   \frac{1}{4}.
 \end{equation}
 Combining \eqref{eq:4.Second_Derivative_Free_Energy.3.3} and
 \eqref{eq:4.Second_Derivative_Free_Energy.3.4}, we arrive at
 \begin{equation}
  \label{eq:4.Second_Derivative_Free_Energy.3.5}
  \begin{split}
   &\quad
   2\int_\Omega
   \biggl(
   1
   -
   \frac{2(1+n)\Cr{const:grad_D}^2(\Cr{const:log_f}+1)^2}{\Cr{const:D_Min}}
   -
   \frac{4D(x)\Cr{const:grad_Pi}^2}{\Cr{const:Pi_Min}^2}
   \biggr)
   D(x)|\nabla \vec{u}|^2 f\,dx
   \\
   &\qquad
   -
   \left(
   \frac{3\Cr{const:log_f}\Cr{const:Pi_Max}\Cr{const:grad_D}}{\Cr{const:D_Min}}
   +
   \frac{3\Cr{const:grad_Pi}}{2}
   \right)
   \Cr{const:4.Sobolev}^{\frac{3}{2}}
   \int_\Omega D(x)|\nabla\vec{u}|^2f\,dx
   \\
   &\geq
   \int_\Omega
   D(x)|\nabla \vec{u}|^2 f\,dx.
  \end{split} 
 \end{equation}
 Finally, we use \eqref{eq:4.Second_Derivative_Free_Energy.3.Assumption1}
 to estimate the coefficient of the last term (the coefficient of the
 integral cubed), and we have that,
 \begin{equation}
  \label{eq:4.Second_Derivative_Free_Energy.3.6}
   \left(
   \frac{\Cr{const:log_f}\Cr{const:Pi_Max}\Cr{const:grad_D}}{2\Cr{const:D_Min}}
   +
   \frac{\Cr{const:grad_Pi}}{4}
   \right)
   \Cr{const:4.Sobolev}^{\frac{3}{2}}
   \leq
   \frac{1}{12}.
 \end{equation}
 Combining \eqref{eq:4.Second_Derivative_Free_Energy.3.2},
 \eqref{eq:4.Second_Derivative_Free_Energy.3.5}, and
 \eqref{eq:4.Second_Derivative_Free_Energy.3.6}, we obtain
 \eqref{eq:4.Second_Derivative_Free_Energy.3}.
\end{proof}

Finally, the  result in Proposition \ref{prop:4.Second_Derivative_Free_Energy} will let us to the 
desired differential inequality for the dissipation term and the proof of the main result Theorem \ref{thm:4}.

\begin{lemma}
 \label{lem:4.Differential_Inequality_Dissipation} 
 Let $f$ be a solution of \eqref{eq:4.FokkerPlanck}, and let $\vec{u}$ be given as
 in \eqref{eq:4.FokkerPlanck}.  Then for any $\gamma>0$, there exist a sufficiently
 large positive constant $\Cr{const:D_Min}\geq1$ and sufficiently small
 positive constants $\Cr{const:grad_Pi}, \Cr{const:Pi_Time}>0$ which depend
 only on $n$, $\|\phi\|_{L^\infty(\Omega)}$,
 $\|\nabla\phi\|_{L^\infty(\Omega)}$, $\lambda$ appeared in
 \eqref{eq:1.Hesse_Potential_Lower_Bounds}(the lower bound of the
 Hessian of $\phi$), $\Cr{const:InitMin}$, $\Cr{const:InitMax}$ defined
 in \eqref{eq:1.Initial}(the bounds of the initial datum $f_0$),
 $\Cr{const:Pi_Min}$, $\Cr{const:Pi_Max}$ defined in
 \eqref{eq:1.PiMinMax}(the bounds of $\pi$), $\Cr{const:grad_D}$ defined
 in \eqref{eq:1.grad_D}(the bound of the gradient of $D$),
 $\Cr{const:4.Sobolev}$ appeared in \eqref{eq:4.SobolevType}, and
 $\gamma$ such that 
 if \eqref{eq:1.D_Min}, \eqref{eq:1.Pi_Time},
 and
\eqref{eq:1.grad_Pi} hold,
 then,
 \begin{equation}
  \label{eq:4.Differential_Inequality_Dissipation}
   \frac{d^2}{dt^2}F[f](t)
   \geq
   \gamma\int_{\Omega}
   \pi(x,t)|\vec{u}|^2
   f\,dx
   -
   \frac{1}{12\Cr{const:Pi_Min}^3}
   \left(
   \int_\Omega
   \pi(x,t)
   |\vec{u}|^2
   f
   \,dx
   \right)^3.
 \end{equation}
\end{lemma}

\begin{proof}
 By Lemma \ref{lem:4.Poincare},
 \begin{equation}
   \int_\Omega|\vec{u}|^2f\,dx
   \leq
   \Cr{const:4.Poincare}
   \int_\Omega
   |\nabla\vec{u}|^{2}
   f\,dx,
 \end{equation} 
 provided the relation  \eqref{eq:4.Poincare_Assumption}: 
$ \Cr{const:grad_Pi}
   \leq
   \frac{\Cr{const:Pi_Min}}{2\Cr{const:4.Sobolev}}.$
With this assumption, 
 by \eqref{eq:1.D_Min} and \eqref{eq:4.Poincare}, we have that,
 \begin{equation}
  \begin{split}
    &
   -2(\lambda+1)
   \int_\Omega
   |\vec{u}|^2
   f
   \,dx
   +
   \int_\Omega
   D(x)|\nabla \vec{u}|^2 f\,dx
    \\
   &\geq
   -2(\lambda+1)
   \int_\Omega
   |\vec{u}|^2
   f
   \,dx
   +
   \Cr{const:D_Min}
   \int_\Omega
   |\nabla \vec{u}|^2 f\,dx
   \\
   &\geq
   \left(
   -2(\lambda+1)
   +
   \frac{\Cr{const:D_Min}}{\Cr{const:4.Poincare}}
   \right)
   \int_\Omega
   |\vec{u}|^2
   f
   \,dx.
  \end{split}
 \end{equation}
 Note that $\Cr{const:4.Poincare}$ depends only on $n$,
 $\Cr{const:InitMin}$, $\Cr{const:InitMax}$, $\Cr{const:Pi_Min}$,
 $\Cr{const:Pi_Max}$, $\Cr{const:grad_D}$, and
 $\|\phi\|_{L^\infty(\Omega)}$, but is independent of
 $\Cr{const:D_Min}$. Thus, for $\gamma>0$, take $\Cr{const:D_Min}\geq1$
 sufficiently large such that
 \begin{equation}
  -2(\lambda+1)
   +
   \frac{\Cr{const:D_Min}}{\Cr{const:4.Poincare}}
 \geq
 \gamma
 \Cr{const:Pi_Max}.
 \end{equation}
 Further, we take $\Cr{const:D_Min}\geq1$ sufficiently large and sufficiently small
 $\Cr{const:grad_Pi}, \Cr{const:Pi_Time}>0$ such that the assumptions
 \eqref{eq:4.Second_Derivative_Free_Energy.3.Assumption1},
 \eqref{eq:4.Second_Derivative_Free_Energy.3.Assumption3},
 \eqref{eq:4.Second_Derivative_Free_Energy.3.Assumption4},
 \eqref{eq:4.Second_Derivative_Free_Energy.3.Assumption5}, and
 \eqref{eq:4.Poincare_Assumption} hold.  Note that
 from Corollary \ref{cor:1.bounds_of_log_f} and Lemma
 \ref{lem:4.Poincare}, $\Cr{const:log_f}$ and $\Cr{const:4.Poincare}$ are
 independent of $\Cr{const:D_Min}$, namely $\log f$ is bounded uniformly
 with respect to $\Cr{const:D_Min}$. Then, we can use
 \eqref{eq:4.Second_Derivative_Free_Energy.3} and
 \begin{equation}
   \frac{d^2}{dt^2}F[f](t)
   \geq
   \gamma
   \Cr{const:Pi_Max}
   \int_\Omega
   |\vec{u}|^2
   f
   \,dx
   -\frac{1}{12}
   \left(
   \int_\Omega
   |\vec{u}|^2
   f
   \,dx
   \right)^3.
 \end{equation}
 Finally, using the bounds of $\pi$ in \eqref{eq:1.PiMinMax}, we obtain
 \eqref{eq:4.Differential_Inequality_Dissipation}.
\end{proof}

Now  we are ready to prove the main result of this Section.

\begin{proof}
 [Proof of Theorem \ref{thm:4}] 
 For any $\gamma>0$, using Lemma
 \ref{lem:4.Differential_Inequality_Dissipation}, take sufficiently
 large positive number $\Cr{const:D_Min}\geq1$ and sufficiently small
 positive numbers $\Cr{const:Pi_Time}, \Cr{const:grad_Pi}>0$. 
 Then define,
 \begin{equation*}
  g(t)=-\frac{d}{dt}F[f](t)
  =
  \int_\Omega 
  \pi(x,t)
  |\vec{u}|^2
  f\,dx,
   \end{equation*} 
   and 
 $ c=\gamma,\quad
  p=3,$
 together with
 $
  d=\frac{1}{12\Cr{const:Pi_Min}^3}.  
$  
 Thus, if
 \begin{equation}
  g(0)
   =
  \int_\Omega
  \pi(x,0)
   |\nabla (D(x)\log f_0(x)+\phi(x))|^2
   f_0\,dx
   <
   \left(
    12\gamma\Cr{const:Pi_Min}^3
   \right)^{\frac{1}{2}},
 \end{equation}
 then by Lemma \ref{lem:3.Gronwall}
 \begin{equation}
    \begin{split}
  g(t)
   &=
   \int_\Omega
   \pi(x,t)
   |\vec{u}|^2f
   \,dx
   \\
   &\leq
   \left(
    \left(
     \int_\Omega
     \pi(x,0)
     |\nabla (D(x)\log f_0(x)+\phi(x))|^2
     f_0\,dx
    \right)^{-2}
    -
    \frac{1}{12\gamma\Cr{const:Pi_Min}^3}
   \right)^{-\frac{1}{2}}
   e^{-\gamma t}.
    \end{split}
 \end{equation}
 Therefore, by taking
 $\Cr{const:4.Initial_Energy}=(12\gamma\Cr{const:Pi_Min}^3)^{\frac{1}{2}}$
 and
 \begin{equation}
  \Cr{const:4.Exponential_Coefficient}
   =
    \left(
    \left(
     \int_\Omega
     \pi(x,0)
     |\nabla (D(x)\log f_0(x)+\phi(x))|^2
     f_0\,dx
    \right)^{-2}
    -
    \frac{1}{12\gamma\Cr{const:Pi_Min}^3}
   \right)^{-\frac{1}{2}},
 \end{equation}
we have finished the proof of  Theorem \ref{thm:4}.
\end{proof}

\section{Conclusion}
\label{sec:5}
In this work, we considered generalized nonlinear Fokker-Planck type equations with
inhomogeneous diffusion and with variable mobility parameters. Such systems
appear as a part of grain growth modeling in polycrystalline
materials. Using new reinterpretation of the classical entropy method
under settings of the bounded domain with periodic boundary
conditions and non-convexity assumptions on the potential function, we obtained the long time behavior of the solutions.


\section*{Acknowledgments}

The work of Yekaterina Epshteyn was partially supported by NSF
DMS-1905463 and by NSF DMS-2118172. The work of Chun Liu was partially
supported by NSF DMS-1950868 and by NSF DMS-2118181. The work of
Masashi Mizuno was partially supported by JSPS KAKENHI Grant Numbers
JP22K03376 and JP23H00085.

\bibliographystyle{plain}
\bibliography{references}

\begin{thebibliography}{10}

\bibitem{MR1842428}
Anton Arnold, Peter Markowich, Giuseppe Toscani, and Andreas Unterreiter.
\newblock On convex {S}obolev inequalities and the rate of convergence to
  equilibrium for {F}okker-{P}lanck type equations.
\newblock {\em Comm. Partial Differential Equations}, 26(1-2):43--100, 2001.

\bibitem{baierlein_1999}
Ralph Baierlein.
\newblock {\em Thermal Physics}.
\newblock Cambridge University Press, 1999.

\bibitem{MR3729587}
Patrick Bardsley, Katayun Barmak, Eva Eggeling, Yekaterina Epshteyn, David
  Kinderlehrer, and Shlomo Ta'asan.
\newblock Towards a gradient flow for microstructure.
\newblock {\em Atti Accad. Naz. Lincei Rend. Lincei Mat. Appl.},
  28(4):777--805, 2017.

\bibitem{DK:gbphysrev}
K.~Barmak, E.~Eggeling, M.~Emelianenko, Y.~Epshteyn, D.~Kinderlehrer, R.~Sharp,
  and S.~Ta'asan.
\newblock Critical events, entropy, and the grain boundary character
  distribution.
\newblock {\em Phys. Rev. B}, 83:134117, Apr 2011.

\bibitem{DK:BEEEKT}
K.~Barmak, E.~Eggeling, M.~Emelianenko, Y.~Epshteyn, D.~Kinderlehrer, and
  S.~Ta'asan.
\newblock Geometric growth and character development in large metastable
  networks.
\newblock {\em Rend. Mat. Appl. (7)}, 29(1):65--81, 2009.

\bibitem{Katya-Chun-Mzn4}
Katayun Barmak, Anastasia Dunca, Yekaterina Epshteyn, Chun Liu, and Masashi
  Mizuno.
\newblock Grain growth and the effect of different time scales.
\newblock In {\em Research in mathematics of materials science}, volume~31 of
  {\em Assoc. Women Math. Ser.}, pages 33--58. Springer, Cham, [2022]
  \copyright 2022.

\bibitem{MR2772123}
Katayun Barmak, Eva Eggeling, Maria Emelianenko, Yekaterina Epshteyn, David
  Kinderlehrer, Richard Sharp, and Shlomo Ta'asan.
\newblock An entropy based theory of the grain boundary character distribution.
\newblock {\em Discrete Contin. Dyn. Syst.}, 30(2):427--454, 2011.

\bibitem{arXiv:2502.13151}
Batuhan Bayir, Yekaterina Epshteyn, and William~M Feldman.
\newblock Global well-posedness of a nonlinear {F}okker-{P}lanck type model of
  grain growth.
\newblock arXiv:2502.13151, 2025.

\bibitem{BA47390682}
R.~Stephen Berry, Stuart~Alan Rice, and John Ross.
\newblock {\em Physical chemistry}.
\newblock Topics in physical chemistry series. Oxford University Press, 2nd ed.
  edition, 2000.

\bibitem{MR2759829}
Haim Brezis.
\newblock {\em Functional analysis, {S}obolev spaces and partial differential
  equations}.
\newblock Universitext. Springer, New York, 2011.

\bibitem{MR3485127}
Jos\'{e}~A. Ca\~{n}izo, Jos\'{e}~A. Carrillo, Philippe Lauren\c{c}ot, and
  Jes\'{u}s Rosado.
\newblock The {F}okker-{P}lanck equation for bosons in 2{D}: well-posedness and
  asymptotic behavior.
\newblock {\em Nonlinear Anal.}, 137:291--305, 2016.

\bibitem{MR1639292}
J.~A. Carrillo and G.~Toscani.
\newblock Exponential convergence toward equilibrium for homogeneous
  {F}okker-{P}lanck-type equations.
\newblock {\em Math. Methods Appl. Sci.}, 21(13):1269--1286, 1998.

\bibitem{MR3019444}
Jos\'{e}~A. Carrillo, Mar\'{\i}a D.~M. Gonz\'{a}lez, Maria~P. Gualdani, and
  Maria~E. Schonbek.
\newblock Classical solutions for a nonlinear {F}okker-{P}lanck equation
  arising in computational neuroscience.
\newblock {\em Comm. Partial Differential Equations}, 38(3):385--409, 2013.

\bibitem{MR4196904}
Pierre Degond, Maxime Herda, and Sepideh Mirrahimi.
\newblock A {F}okker-{P}lanck approach to the study of robustness in gene
  expression.
\newblock {\em Math. Biosci. Eng.}, 17(6):6459--6486, 2020.

\bibitem{MR3932086}
Weinan E, Tiejun Li, and Eric Vanden-Eijnden.
\newblock {\em Applied stochastic analysis}, volume 199 of {\em Graduate
  Studies in Mathematics}.
\newblock American Mathematical Society, Providence, RI, 2019.

\bibitem{MR4506846}
Yekaterina Epshteyn, Chang Liu, Chun Liu, and Masashi Mizuno.
\newblock Nonlinear inhomogeneous {F}okker-{P}lanck models:
  energetic-variational structures and long-time behavior.
\newblock {\em Anal. Appl. (Singap.)}, 20(6):1295--1356, 2022.

\bibitem{epshteyn2022local}
Yekaterina Epshteyn, Chang Liu, Chun Liu, and Masashi Mizuno.
\newblock Local well-posedness of a nonlinear {F}okker-{P}lanck model.
\newblock {\em Nonlinearity}, 36(3):1890, feb 2023.

\bibitem{CMS-Katya-Chun-Masashi}
Yekaterina Epshteyn, Chun Liu, and Masashi Mizuno.
\newblock Large time asymptotic behavior of grain boundaries motion with
  dynamic lattice misorientations and with triple junctions drag.
\newblock {\em Communications in Mathematical Sciences}, 19(5):1403--1428,
  2021.

\bibitem{MR4263432}
Yekaterina Epshteyn, Chun Liu, and Masashi Mizuno.
\newblock Motion of {G}rain {B}oundaries with {D}ynamic {L}attice
  {M}isorientations and with {T}riple {J}unctions {D}rag.
\newblock {\em SIAM J. Math. Anal.}, 53(3):3072--3097, 2021.

\bibitem{epshteyn2021stochastic}
Yekaterina Epshteyn, Chun Liu, and Masashi Mizuno.
\newblock A stochastic model of grain boundary dynamics: a {F}okker-{P}lanck
  perspective.
\newblock {\em Math. Models Methods Appl. Sci.}, 32(11):2189--2236, 2022.

\bibitem{MR1607500}
J.~L. Ericksen.
\newblock {\em Introduction to the thermodynamics of solids}, volume 131 of
  {\em Applied Mathematical Sciences}.
\newblock Springer-Verlag, New York, revised edition, 1998.

\bibitem{MR1625845}
Lawrence~C. Evans.
\newblock {\em Partial differential equations}, volume~19 of {\em Graduate
  Studies in Mathematics}.
\newblock American Mathematical Society, Providence, RI, 1998.

\bibitem{MR2053476}
C.~W. Gardiner.
\newblock {\em Handbook of stochastic methods for physics, chemistry and the
  natural sciences}, volume~13 of {\em Springer Series in Synergetics}.
\newblock Springer-Verlag, Berlin, third edition, 2004.

\bibitem{MR3916774}
Mi-Ho Giga, Arkadz Kirshtein, and Chun Liu.
\newblock Variational modeling and complex fluids.
\newblock In {\em Handbook of mathematical analysis in mechanics of viscous
  fluids}, pages 73--113. Springer, Cham, 2018.

\bibitem{MR1814364}
David Gilbarg and Neil~S. Trudinger.
\newblock {\em Elliptic partial differential equations of second order}.
\newblock Classics in Mathematics. Springer-Verlag, Berlin, 2001.
\newblock Reprint of the 1998 edition.

\bibitem{MR2467561}
Timothy Gowers, June Barrow-Green, and Imre Leader, editors.
\newblock {\em The {P}rinceton companion to mathematics}.
\newblock Princeton University Press, Princeton, NJ, 2008.

\bibitem{MR4218540}
Jingwei Hu, Jian-Guo Liu, Yantong Xie, and Zhennan Zhou.
\newblock A structure preserving numerical scheme for {F}okker-{P}lanck
  equations of neuron networks: numerical analysis and exploration.
\newblock {\em J. Comput. Phys.}, 433:Paper No. 110195, 23, 2021.

\bibitem{MR3497125}
Ansgar J\"ungel.
\newblock {\em Entropy methods for diffusive partial differential equations}.
\newblock SpringerBriefs in Mathematics. Springer, [Cham], 2016.

\bibitem{MR0241822}
O.~A. Lady{\v{z}}enskaja, V.~A. Solonnikov, and N.~N. Ural'ceva.
\newblock {\em Linear and quasilinear equations of parabolic type}.
\newblock Translated from the Russian by S. Smith. Translations of Mathematical
  Monographs, Vol. 23. American Mathematical Society, Providence, R.I., 1967.

\bibitem{MR1465184}
Gary~M. Lieberman.
\newblock {\em Second order parabolic differential equations}.
\newblock World Scientific Publishing Co. Inc., River Edge, NJ, 1996.

\bibitem{MR2165379}
Fang-Hua Lin, Chun Liu, and Ping Zhang.
\newblock On hydrodynamics of viscoelastic fluids.
\newblock {\em Comm. Pure Appl. Math.}, 58(11):1437--1471, 2005.

\bibitem{MR1812873}
P.~A. Markowich and C.~Villani.
\newblock On the trend to equilibrium for the {F}okker-{P}lanck equation: an
  interplay between physics and functional analysis.
\newblock In {\em VI Workshop on Partial Differential Equations, Part II (Rio
  de Janeiro, 1999)}, volume~19, pages 1--29. Sociedade Brasileira de
  Matem\'{a}tica, Rio de Janeiro, 2000.

\bibitem{BA00323160}
Donald~Allan McQuarrie.
\newblock {\em Statistical mechanics}.
\newblock Harper's chemistry series. Harper \& Row, 1976.

\bibitem{MR3288096}
Grigorios~A. Pavliotis.
\newblock {\em Stochastic processes and applications}, volume~60 of {\em Texts
  in Applied Mathematics}.
\newblock Springer, New York, 2014.
\newblock Diffusion processes, the Fokker-Planck and Langevin equations.

\bibitem{MR987631}
H.~Risken.
\newblock {\em The {F}okker-{P}lanck equation}, volume~18 of {\em Springer
  Series in Synergetics}.
\newblock Springer-Verlag, Berlin, second edition, 1989.
\newblock Methods of solution and applications.

\bibitem{MR0274683}
R.~Tyrrell Rockafellar.
\newblock {\em Convex analysis}, volume No. 28 of {\em Princeton Mathematical
  Series}.
\newblock Princeton University Press, Princeton, NJ, 1970.

\bibitem{MR1839500}
Silvio R.~A. Salinas.
\newblock {\em Introduction to statistical physics}.
\newblock Graduate Texts in Contemporary Physics. Springer-Verlag, New York,
  2001.
\newblock Translated from the Portuguese.

\bibitem{PhysRevE.94.062117}
Gabriele Sicuro, Peter Rap\ifmmode~\check{c}\else \v{c}\fi{}an, and Constantino
  Tsallis.
\newblock Nonlinear inhomogeneous fokker-planck equations: Entropy and
  free-energy time evolution.
\newblock {\em Phys. Rev. E}, 94:062117, Dec 2016.

\bibitem{MR4439423}
Yiwei Wang and Chun Liu.
\newblock Some recent advances in energetic variational approaches.
\newblock {\em Entropy}, 24(5):Paper No. 721, 26, 2022.

\bibitem{MR3021544}
Hao Wu, Xiang Xu, and Chun Liu.
\newblock On the general {E}ricksen-{L}eslie system: {P}arodi's relation,
  well-posedness and stability.
\newblock {\em Arch. Ration. Mech. Anal.}, 208(1):59--107, 2013.

\end{thebibliography}

\end{document}